\documentclass[a4paper, 11pt, oneside]{article}
\usepackage[utf8]{inputenc}
\usepackage{graphicx, subfigure}
\usepackage{cite, ae, alltt, framed, verbatim}
\usepackage{amsmath, amssymb, amsthm, amsfonts, amsxtra}
\usepackage{mathrsfs, upgreek} %\usepackage[mathcal]{euscript}
\usepackage[usenames,dvipsnames]{color}
\definecolor{old}{gray}{0.25}
\usepackage{accents}
\usepackage{bm, algpseudocode}
\usepackage[runin]{abstract}
\usepackage{hyperref}
\hypersetup{
  pdftitle={Error control SPDE},
  pdfauthor={},
  pdfsubject={},
  urlcolor=black,
  pdfdisplaydoctitle=true, pdfduplex=DuplexFlipLongEdge, pdfmenubar=true, pdfview=FitH, pdfstartview=FitV, colorlinks=true, linkcolor=black, menucolor=black, filecolor=black, citecolor=black
}
\usepackage[all]{nowidow}

\let\oldmarginpar\marginpar
\renewcommand\marginpar[1]{%
\-\oldmarginpar[\raggedleft\footnotesize{\color{red} #1}]%
{\raggedright\footnotesize{\color{red} #1}}}

\newtheoremstyle{Theorem}%
  {3pt}%                  space above
  {3pt}%                  space below
  {\itshape}%              body font {\normalfont}
  {}%                     indent or {\parindent}
  {\normalfont\bfseries}% head font
  {.}%                    head punctuation
  {0.5em}%                head space
  {{\normalfont\textbf{\thmname{#1}\thmnumber{ #2}}\thmnote{ #3}}}
\newtheoremstyle{Definition}%
  {3pt}%                  space above
  {3pt}%                  space below
  {}%                     body font {\normalfont}
  {}%                     indent or {\parindent}
  {\normalfont\bfseries}% head font
  {.}%                    head punctuation
  {0.5em}%                head space
  {{\normalfont\textbf{\thmname{#1}\thmnumber{ #2}}\thmnote{ #3}}}
\theoremstyle{Theorem}
\newtheorem{Theorem}{Theorem}[section]
\newtheorem*{Theorem*}{Theorem}
\newtheorem{Proposition}[Theorem]{Proposition}

\newtheorem{Lemma}[Theorem]{Lemma}

\theoremstyle{Definition}
\newtheorem{Remark}[Theorem]{Remark}
\newtheorem*{Remark*}{Remark}

\newtheorem{Definition}[Theorem]{Definition}
\newtheorem{Example}[Theorem]{Example}
\newtheorem{Assumption}[Theorem]{Assumption}

\renewcommand{\epsilon}{\varepsilon}
\renewcommand{\rho}{\varrho}
\newcommand{\md}{\ensuremath{\,\textup{d}}}
\newcommand{\domain}{\ensuremath{\mathcal O}}
\newcommand{\hs}{\ensuremath{\mathcal L_{\text{HS}}}}
\newcommand{\Lip}[2]{\ensuremath{C^{\textup{Lip}}_{\tau, {#1}, {#2}}}}
\newcommand{\R}{\mathbb{R}}

\newcommand{\N}{\mathbb{N}}

\renewcommand{\H}{\ensuremath{\mathcal{H}}}
\newcommand{\U}{\ensuremath{U}}
\newcommand{\G}{\ensuremath{\mathcal{G}}}
\author{P.A. Cioica, S. Dahlke, N. D\"{o}hring, U. Friedrich, S. Kinzel, \\F. Lindner, T. Raasch, K. Ritter, R.L. Schilling}
\title{On the convergence analysis of the inexact linearly implicit Euler scheme for a class of SPDEs\thanks{This work has been supported by the Deutsche Forschungsgemeinschaft (DFG, grants DA 360/12-2, DA 360/13-2, DA 360/20-1, RI 599/4-2, SCHI 419/5-2) and a doctoral scholarship of the Philipps-Universit\"{a}t Marburg.}}

\date{}

\begin{document}

\frenchspacing

\maketitle
\begin{abstract} This paper is concerned with the adaptive numerical treatment of stochastic partial differential equations. Our method of choice is Rothe's method. We use the implicit Euler scheme for the time discretization. Consequently, in each step,  an  elliptic equation with random right-hand side has to be solved.  In practice, this  cannot be performed exactly, so that efficient numerical methods are needed. Well-established adaptive wavelet or finite-element schemes, which are guaranteed to converge with optimal order, suggest themselves. 
We investigate  how the errors corresponding to the adaptive spatial discretization propagate  in time,  and we show how  in each time step the tolerances
have to be chosen such that the resulting  perturbed discretization scheme realizes the same order of convergence 
%(in time direction) 
as the one  with exact evaluations of the elliptic subproblems. 
\end{abstract}


\bigskip

\section{Introduction}\label{sec:introduction}

This paper is concerned with the numerical treatment of stochastic evolution equations of the form 
\begin{equation}
\label{SPDE}
\md u(t) = \big(A u(t) + f(u(t))\big) \md t + B(u(t)) \md W(t), \quad u(0) = u_0,
\end{equation}
on the time interval $[0,T]$ in a real and separable Hilbert space $U$. Here, $A:D(A)\subset \U\to\U$ is a densely defined, strictly negative definite, self-adjoint, linear operator such that zero belongs to the resolvent set and the inverse $A^{-1}$ is compact on $U$. The forcing terms $f:D((-A)^\rho)\to D((-A)^{\rho-\sigma})$ and $B: D((-A)^\rho)\to \mathcal L(\ell_2,D((-A)^{\rho-\beta}))$ are Lipschitz continuous maps for suitable constants $\rho$, $\sigma$ and $\beta$; and finally,  $W=(W(t))_{t\in[0,T]}$ is a cylindrical Wiener process on the sequence space $\ell_2=\ell_2(\mathbb N)$. 
In practical applications, evolution equations of the form \eqref{SPDE} are abstract formulations of stochastic partial differential equations (SPDEs, for short): Usually $A$ is a differential operator, $f$ a linear or nonlinear forcing term and $B(u(t)) \md W(t)$ describes additive or multiplicative noise.
They are models, e.g., for reaction diffusion processes corrupted by noise, which are frequently used for the mathematical description of biological, chemical and physical processes. 
Details on the equation, the operators $A$, the forcing terms $f$ and $B$ and the initial condition $u_0$ are given in Section~\ref{sec:setting-and-assumptions}.
Usually, the exact solution of (\ref{SPDE}) cannot be computed explicitly, so that  numerical schemes for the constructive approximation of the solutions are needed. For stochastic parabolic equations,   there are two principally different approaches: the vertical method of lines and the horizontal method of lines. The former 
starts with an approximation first in space and then in time. We refer to \cite{GyoMil2009, GyoKry2010, Gyo2014,  Hal2012} for detailed information. The  latter 
starts with a discretization first in time and then in space; it is also known as Rothe's method. 
In the stochastic setting, it has been studied, e.g., in \cite{bib:GrTu95, bib:BrGr97}. These references are indicative and by no means complete. 

Very often, the vertical method of lines is preferred since, at first sight, it seems to be a little bit simpler.  Indeed, after the disretization in space is performed,  just an ordinary finite dimensional stochastic differential equation (SDE, for short) in time direction has to be solved, and there exists a huge amount of approaches for the numerical treatment of SDEs. However, there are also certain drawbacks.  In many applications, the   utilization  of adaptive strategies to increase efficiency is almost unavoidable.  In the context of the vertical method of lines, the combination with spatial adaptivity is at least not straightforward. In contrast,  for the horizontal   method of lines,  the following natural approach suggests itself. 
Using Rothe's method, the SPDE can be interpreted as an abstract Cauchy problem, i.e., as a stochastic differential equation in some suitable function spaces.  Then,  in time direction we might use an SDE-solver with step size control. This solver must be based on an implicit discretization scheme since the equation under consideration is usually stiff. Consequently, in each time step, an elliptic equation with random right-hand side has to be solved.
To this end, as a second level of adaptivity, adaptive numerical schemes that are well-established for deterministic equations, can be used. We refer to \cite{bib:CDD01, bib:CDD02, bib:DFRSW07} for suitable wavelet methods, and to \cite{bib:Verfuerth94, bib:Verfuerth96, bib:HansboJohnson92, bib:BabuskaRheinboldt83, bib:Babuska88, bib:BankWeiser85, bib:Erik94, bib:ErJo91, bib:ErJo95, bib:ErJoLa98, bib:BoErKo96} 
for the finite element case. As before, these lists are not complete. 

Although this combination   with adaptive strategies is natural, the mathematical  analysis of the resulting  schemes  seems to be still in its infancy. In the stochastic setting, Rothe's method with exact evaluation of the elliptic subproblems, 
has been considered, e.g., in \cite{bib:BrGr97, bib:GrTu95}, and explicit convergence rates have been established, e.g., in \cite{bib:Cox2012, bib:CVN2012, bib:GyNu97, bib:Printems2001}.  
First results concerning  the combination with adaptive space discretization methods
based on wavelets  have been shown in \cite{bib:KovLarUrb2013}. 

Even  for the deterministic case, not many results concerning a rigorous convergence and complexity analysis of the overall scheme seem to be available.  To our best knowledge, the most far reaching achievements  have been obtained in \cite{bib:CDDFKLRRS2014}.  In this paper, it has been clarified how the tolerances for the elliptic subproblems in each time step have to be tuned so that the overall (perturbed)  discretization scheme realizes the same order of convergence (in time direction) as the unperturbed one. Moreover, based on concepts from  approximation theory and function space theory, respectively,  a complexity analysis of the overall scheme has been derived. 
It is the aim of this paper to generalize the analysis presented in \cite{bib:CDDFKLRRS2014} to SPDEs of the form~\eqref{SPDE}.  
We mainly consider the case of the implicit Euler scheme, and we concentrate on the convergence analysis. 
To our best knowledge, no result in this direction has been reported yet. 
Complexity estimates are beyond  the scope of this work  and will be presented in a forthcoming paper.

For reader's convenience, let us briefly recall the basic approach of \cite{bib:CDDFKLRRS2014} for the deterministic case, confined to the implicit Euler scheme. 
As a typical example, let us consider the deterministic heat equation 
\begin{equation*}%\label{eq:example-heat-equation}
%\left.
\begin{aligned}
 u^\prime (t) &= \Delta u(t) + f(t, u(t)) \quad &&\text{on } \domain,\ t\in(0,T],  \\
 u &= 0 \quad &&\text{on } \partial\domain,\ t\in(0,T],\\ 
 u(0) &= u_0\quad &&\text{on } \domain, 
\end{aligned}
%\qquad \right\}
\end{equation*}
where $\domain\subset\R^d$, $d\geq 1$, denotes a bounded Lipschitz domain. 
We discretize this equation by means of a linearly implicit Euler scheme with uniform time steps. Let $K\in\mathbb{N}$ be the number of subdivisions of the time interval $[0,T]$. The step size will be denoted by $\tau := T/K$, and the $k$-th point in time is denoted by $t_k:=\tau k$, $k \in\{0, \ldots, K\}$. The linearly implicit Euler scheme, starting at $u_0$, is given by 
\begin{equation*} %\label{eq:example_leads-to-euler-scheme}
\frac{u_{k+1} - u_k}{\tau} = \Delta u_{k+1} + f(t_k, u_k),
\end{equation*}
i.e.,
\begin{equation} \label{eq:example_the-linear-euler-scheme}
(I - \tau \Delta) u_{k+1} = u_k + \tau f(t_k, u_k),
\end{equation}
for $k=0,\ldots,K-1$. If we assume that the elliptic problem
\begin{equation*} %\label{eq:example_elliptic-equation}
L_{\tau} v := (I - \tau \Delta) v = g \quad \text{on } \domain, \quad v|_{\partial \domain} = 0, 
\end{equation*}
can be solved exactly, then one step of the scheme \eqref{eq:example_the-linear-euler-scheme} can be written as
\begin{equation} \label{eq:example_base-of-abstract-euler-scheme}
u_{k+1} = L_{\tau}^{-1}R_{\tau,k}(u_k),
\end{equation}
where
\begin{equation*} %\label{eq:def-von-R}
R_{\tau, k}(w) := w + \tau f(t_k, w)
\end{equation*}
and $L_\tau$ is a boundedly invertible operator 
between suitable Hilbert spaces. 

In practice, the elliptic problems in~\eqref{eq:example_base-of-abstract-euler-scheme} cannot be evaluated exactly. 
Instead, we employ a `black box' numerical scheme, which for any prescribed tolerance $\varepsilon>0$  yields an approximation $[v]_\varepsilon$ of $v:=L_\tau^{-1}R_{\tau,k}(w)$, where $w$ is an element of a suitable Hilbert space, i.e.,
\begin{equation*}%\label{eq:intro-epsilons}
\lVert v-[v]_\varepsilon\rVert\leq \epsilon,
\end{equation*}
for a proper norm $\lVert\cdot\rVert$.  
What we have in mind are applications of adaptive wavelet solvers, which are guaranteed to converge with optimal order, as developed, e.g., in \cite{bib:CDD01}, combined with efficient evaluations of the nonlinearities $f$ as they can be found, e.g., in \cite{bib:Kap2011, bib:DahSchXu2000, bib:CohDahDeV2000}. 
In \cite{bib:CDDFKLRRS2014} we have  investigated  how  the error propagates within the linearly implicit Euler scheme and how the tolerances $\epsilon_k$ in each time step have to be chosen, such that we obtain the same order of convergence as in the case of exact evaluation of the elliptic problems. We have shown that the tolerances depend on the Lipschitz constants $\Lip jk$ of the operators \begin{equation*}
E_{\tau,j,k}=(L_\tau^{-1} R_{\tau,k-1}) \circ (L_\tau^{-1} R_{\tau,k-2}) \circ \dots \circ (L_\tau^{-1} R_{\tau, j}),
\end{equation*}
with $1\leq j\leq k\leq K$, $K\in\N$, via
\begin{equation*}
\lVert u(t_k) - \tilde{u}_k\rVert \leq \lVert u(t_k) - u_k\rVert + \sum_{j=0}^{k-1} \Lip{j+1}{k} \, \epsilon_{j},
\end{equation*}
where $\tilde u_k$ is the solution to the inexactly evaluated Euler scheme at time $t_k$.

Now let us come back to  SPDEs of the form (\ref{SPDE}).  Once again, for the (adaptive) numerical treatment of (\ref{SPDE}) we  consider for $K\in\mathbb N$ and $\tau:=T/K$ the linearly implicit Euler scheme
\begin{equation}\label{eq:EulerSPDE}
\left.
\begin{alignedat}{2}
u_{k+1}&=(I-\tau A)^{-1}&\big(u_k+\tau f(u_k)+\sqrt\tau B(u_k)\chi_{k}\big),\\
&&k=0,\ldots,K-1,
\end{alignedat}\qquad
\right\}
\end{equation}
with
\begin{equation*}%\label{eq:stoch-incr}
\chi_{k}:=\chi_{k}^K:=\frac 1{\sqrt\tau}\left(W\big(t_{k+1}^K\big)-W\big(t_k^K\big)\right),
\end{equation*}
where 
$t_k:=\tau k$, $k=0,\ldots,K$.
If we set 
\begin{align*}%\label{eq:intro-SPDEinterpretation1}
R_{\tau,k}(w)&:= w+\tau f(w)+\sqrt{\tau} B(w)\chi_{k},&\quad& k=0,\ldots,K-1,\\
L_{\tau}^{-1} w&:=(I-\tau A)^{-1}w,&& k=1,\ldots,K,
\end{align*}
the operators being defined between suitable Hilbert spaces $\mathcal{H}_k$ and $\mathcal{G}_k$, the scheme~\eqref{eq:EulerSPDE} can  again be rewritten as
\begin{equation}\label{eq:intro-short}
u_{k+1}=L_{\tau}^{-1}R_{\tau,k}(u_k),\quad k=0,\ldots,K-1.
\end{equation}
We refer to Section~\ref{sec:semi-discretization-in-time} for a precise formulation of this scheme.

Once again  the elliptic problems in~\eqref{eq:intro-short} cannot  be evaluated exactly.
Similar to the deterministic setting,
we assume that we have at hand a `black box' numerical scheme, which for any required $w$ approximates 
\[
v:=(I-\tau A)^{-1} \big(w + f(w) + B(w)\chi_k \big)
\]
with a prescribed tolerance $\varepsilon>0$.
What we have in mind are applications of some deterministic solver for elliptic equations to individual realizations, e.g., an optimal adaptive wavelet solver as developed in \cite{bib:CDD01}, combined with proper evaluations of the nonlinearities $f$ and $B$, see, e.g., \cite{bib:Kap2011, bib:DahSchXu2000, bib:CohDahDeV2000}, and an adequate truncation of the noise.
It is the aim of this paper to  investigate  how  the error propagates within the linearly implicit inexact Euler scheme for SPDEs (cf.~Proposition~\ref{prop:OverallErrorEstimateSPDE}) and how the tolerances $\epsilon_k$ in each time step have to be chosen, such that we obtain the same order of convergence (in time direction) for the inexact scheme as for its exact counterpart (cf.~Theorem~\ref{thm:central-theorem-in-stochastic-section}). 

Concerning the setting, we follow \cite{bib:Printems2001} and impose rather restrictive conditions on the different parts of Eq.~\eqref{SPDE}. 
This allows us to focus on our main goal, i.e., the analysis of the error of the inexact counterpart of the Euler scheme~\eqref{eq:EulerSPDE}, without spending too much time on explaining details regarding the underlying setting, cf.~Remark~\ref{rem:Setting}.
Compared with \cite{bib:Printems2001} we allow the spatial regularity of the whole setting to be `shifted' in terms of the additional parameter $\rho$. In concrete applications to parabolic SPDEs, this will lead to estimates of the discretization error in terms of the numerically important energy norm, cf.\ Example~\ref{ConcreteExampleSPDE2}, provided that the initial condition $u_0$ and the forcing terms $f$ and $B$ are sufficiently regular.

A different approach has been presented in \cite{bib:KovLarUrb2013}, where additive noise is considered, a splitting method is applied, and adaptivity is only used for the deterministic part of the equation. We remark that the use of spatially adaptive schemes is useful especially for stochastic equations, where singularities appear naturally near the boundary due to the irregular behaviour of the noise, cf. \cite{bib:CioDahDoe+2014b} and the references therein.
 
We choose the following outline. In Section~\ref{sec:setting-and-assumptions} we present the setting and some examples of equations that fit into this setting. In Section~\ref{sec:semi-discretization-in-time} we show how to reformulate the linearly implicit Euler scheme as an abstract Rothe scheme and derive convergence rates under the assumption that we can evaluate the subproblems~\eqref{eq:intro-short} exactly. We drop this assumption in Section~\ref{sec:discretization-in-time-and-space} and focus on how to choose the tolerances for each subproblem, such that we can achieve the same order of convergence.

\section{Setting}
\label{sec:setting-and-assumptions}

In this section we describe the underlying setting in detail. It coincides with the one in~\cite{bib:Printems2001} (`shifted' by $\rho\geq 0$). Furthermore we define the solution concept under consideration and give some examples of equations, which fit into this setting.

We start with assumptions on the linear operator in Eq.~\eqref{SPDE}.

\begin{Assumption}\label{AssumptionASPDE}
The operator $A:D(A)\subset \U\to\U$ is linear, densely defined, strictly negative definite and self-adjoint. Zero belongs to the resolvent set of $A$ and the inverse $A^{-1}:\U\to\U$ is compact. There exists an $\alpha>0$ such that $(-A)^{-\alpha}$ is a trace class operator on $\U$.
\end{Assumption}

To simplify notation, the separable real Hilbert space $\U$ is always assumed to be infinite-dimensional.
Under the assumption above, it follows that $A$ enjoys a spectral decomposition of the form
\begin{equation}
\label{eq:spectralDecomposition}
Av=\sum_{j\in\mathbb{N}}\lambda_j\langle v,e_j\rangle_U e_j,\qquad v\in D(A),
\end{equation}
where $(e_j)_{j\in\mathbb{N}}$ is an orthonormal basis of $\U$ consisting of eigenvectors of $A$ with strictly negative eigenvalues $(\lambda_j)_{j\in\mathbb{N}}$ such that
\begin{equation}
\label{eq:eigenvalues}
0>\lambda_1\geq\lambda_2\geq\ldots\geq\lambda_j\to-\infty,\qquad j\to\infty.
\end{equation}

For $s\geq 0$ we set
\begin{align}
D((-A)^s)&:=\Big\{v\in \U \,:\;\sum_{j=1}^\infty\big|(-\lambda_j)^{s}\langle v, e_j\rangle_U\big|^2<\infty\Big\},\label{eq:defDomain}\\
(-A)^s v&:=\sum_{j\in\N}(-\lambda_j)^s\langle v, e_j\rangle_U e_j,\;\quad \,v\in D((-A)^s),\label{(-A)^s}
\end{align}
so that $D((-A)^s)$, endowed with the norm $\lVert\,\cdot\,\rVert_{D((-A)^s)}:=\lVert(-A)^s\,\cdot\,\rVert_U$, is a Hilbert space; by construction this norm is equivalent to the graph norm of $(-A)^s$.

For $s<0$ we define $D((-A)^s)$ as the completion of $\U$ with respect to the norm $\lVert\,\cdot\,\rVert_{D((-A)^s)}$, defined on $U$ by $\|v\|^2_{D((-A)^s)}:= \sum_{j\in\N}\big|(-\lambda_j)^{s}\langle v, e_j\rangle_U\big|^2$. Thus, $D((-A)^{s})$ can be considered as a space of formal sums
\begin{equation*}
v=\sum_{j\in\N} v^{(j)} e_j,\quad\text{ such that }\quad \sum_{j\in\N}\big|(-\lambda_j)^{s}v^{(j)}\big|^2<\infty
\end{equation*}
with coefficients $v^{(j)}\in\mathbb R$.\ Generalizing \eqref{(-A)^s} in the obvious way, we obtain operators $(-A)^s$, $s\in\mathbb R$, which map $D((-A)^{r})$ isometrically onto $D((-A)^{r-s})$ for all $r\in\mathbb R$.

The trace class condition in Assumption~\ref{AssumptionASPDE} can now be reformulated as the requirement that there exists an $\alpha>0$ such that
\begin{equation}\label{eq:trace}
\operatorname{Tr}(-A)^{-\alpha}=\sum_{j\in\mathbb{N}}(-\lambda_j)^{-\alpha}<\infty.
\end{equation}
Note that any linear operator with a spectral decomposition as in \eqref{eq:spectralDecomposition} and eigenvalues as in \eqref{eq:eigenvalues} and \eqref{eq:trace} fulfills Assumption~\ref{AssumptionASPDE}.
Let us consider a prime example of such an operator. 
Throughout this paper, we write $L_2(\domain)$ for the space of quadratically Lebesgue-integrable real-valued functions on a Borel-measurable subset $\domain$ of $\mathbb R^d$.  
Furthermore, $\mathcal{L}(U_1;U_2)$ stands for the space of bounded linear operators between two Hilbert spaces $U_1$ and $U_2$. If the Hilbert spaces coincide, we simply write $\mathcal{L}(U_1)$ instead of $\mathcal{L}(U_1;U_1)$. 
\begin{Example}\label{example:Operator}
Let $\domain$ be a bounded open subset of $\mathbb R^d$, set $\U:=L_2(\domain)$ and let $A:=\Delta_\domain^D$ be the Dirichlet-Laplacian on $\domain$, i.e.,
\begin{equation*}
\Delta_\domain^D : D(\Delta_\domain^D)\subseteq L_2(\domain)\to L_2(\domain)
\end{equation*}
with domain
\begin{equation*}
D(\Delta_\domain^D)=\Big\{u\in H^1_0(\domain) : \Delta u := \sum_{i=1}^d \frac{\partial^2}{\partial x_i^2} u\in L_2(\domain)\Big\},
\end{equation*}
where $H^1_0(\domain)$ stands for the completion in the $L_2(\domain)$-Sobolev space of order one of the set $\mathcal{C}_0^\infty(\domain)$ of infinitely differentiable functions with compact support in $\domain$.
Note that this definition of the domain of the Dirichlet-Laplacian is consistent with the definition of $D((-\Delta_\domain^D)^s)$ for $s=1$ in \eqref{eq:defDomain}, see, e.g., \cite[Remark~1.13]{bib:LindnerPhD} for details.
This linear operator fulfills Assumption~\ref{AssumptionASPDE} for all $\alpha>d/2$: It is well-known that it is densely defined, self-adjoint, and strictly negative definite. 
Furthermore it possesses a compact inverse $(\Delta_\domain^D)^{-1} : L_2(\domain)\to L_2(\domain)$, see, e.g.,\ \cite{bib:Evans}.
Moreover,  Weyl's law states that
\begin{equation*}
-\lambda_j \asymp j^{2/d},\qquad j\in\mathbb N,
\end{equation*}
see \cite{bib:BirSol}, implying that \eqref{eq:trace} holds for all $\alpha>d/2$.
\end{Example}

Next we state the assumptions on the forcing terms $f$ and $B$.

\begin{Assumption}\label{AssumptionfBSPDE}
For certain smoothness parameters
\begin{equation}\label{eq:rsbeta}
\rho\geq0,\qquad \sigma<1\quad\text{ and }\quad \beta <\frac{1-\alpha}2
\end{equation}
($\alpha$ as in Assumption~\ref{AssumptionASPDE}), we have
\begin{align*}
f : D((-A)^\rho)&\to D((-A)^{\rho-\sigma}),\\
B : D((-A)^\rho)&\to \mathcal L(\ell_2; D((-A)^{\rho-\beta})).
\end{align*}
Furthermore, $f$ and $B$ are globally Lip\-schitz continuous, that is, there exist positive constants $C_f^{\text{Lip}}$ and $C_B^{\text{Lip}}$ such that for all $v$, $w\in D((-A)^\rho)$,
\begin{equation*}
%\label{eq:assf}
\begin{aligned}
\lVert f(v)-f(w)\rVert_{D((-A)^{\rho-\sigma})}&\leq C_f^{\text{Lip}}\lVert v-w\rVert_{D((-A)^\rho)},\\
\end{aligned}
\end{equation*}
and
\begin{equation*}
\begin{aligned}
\lVert B(v)-B(w)\rVert_{\mathcal L(\ell_2; D((-A)^{\rho-\beta}))}&\leq C_B^{\text{Lip}}\lVert v-w\rVert_{D((-A)^\rho)}.
\end{aligned}
\end{equation*}
\end{Assumption}

\begin{Remark}\label{rem:AssumptionBSPDE}
\textit{(i)} The parameters $\sigma$ and $\beta$ in Assumption~\ref{AssumptionfBSPDE} are allowed to be negative.

\noindent\textit{(ii)} Assumption~\ref{AssumptionfBSPDE} follows the lines of \cite{bib:Printems2001} (`shifted' by $\rho\geq0$). The linear growth conditions (3.5) and (3.7) therein follow from the (global) Lipschitz continuity of the mappings $f$ and $B$.
\end{Remark}

Finally, we describe the noise and the initial condition in Eq.~\eqref{SPDE}. For the notion of a normal filtration we refer to \cite{bib:PrRoe07}.

\begin{Assumption}\label{Assumptionu0BSPDE}
The noise $W=(W(t))_{t\in[0,T]}$ is a cylindrical Wiener process on $\ell_2$ with respect to a normal filtration $(\mathscr F_t)_{t\in[0,T]}$. The underlying probability space $(\Omega,\mathscr F,\mathbb P)$ is complete. For $\rho$ as in Assumption~\ref{AssumptionfBSPDE}, the initial condition $u_0$ in Eq.~\eqref{SPDE} satisfies
\begin{equation*}
u_0\in L_2(\Omega,\mathscr F_0,\mathbb P;D((-A)^\rho)).
\end{equation*}
\end{Assumption}

In this paper we consider a mild solution concept. To this end let $(e^{tA})_{t\geq0}$ be the strongly continuous semigroup of contractions on $U$ generated by $A$. 

\begin{Definition}
\label{def:mildSolution}
A \emph{mild solution} to Eq.~\eqref{SPDE} (in $D((-A)^\rho)$) is a predictable process $u:\Omega\times[0,T]\to D((-A)^\rho)$ with
\begin{equation}\label{eq:mildSolutionCond}
\sup_{t\in[0,T]}\mathbb E\|u(t)\|_{D((-A)^\rho)}^2 <\infty,
\end{equation}
 such that for every $t\in[0,T]$ the equality
\begin{equation}\label{eq:mildSolution}
u(t)=e^{tA}u_0+\int_0^t e^{(t-s)A}f(u(s))\md s+\int_0^t e^{(t-s)A}B(u(s))\md W(s)
\end{equation}
holds $\mathbb P$-almost surely in $D((-A)^\rho)$.
\end{Definition}

\begin{Remark} 
\textit{(i)}
Let $u:\Omega\times[0,T]\to D((-A)^\rho)$ be a predictable process fulfilling \eqref{eq:mildSolutionCond}. Then, the first integral in \eqref{eq:mildSolution} is meant to be a $D((-A)^\rho)$-valued Bochner integral for $\mathbb P$-almost every $\omega\in\Omega$; the second integral is a $D((-A)^\rho)$-valued stochastic integral as defined, e.g., in \cite{bib:DPZ92,bib:PrRoe07}. Both integrals exist due to \eqref{eq:mildSolutionCond} and Assumptions~\ref{AssumptionASPDE} and \ref{AssumptionfBSPDE}. For example, considering the stochastic integral in \eqref{eq:mildSolution}, we know that it exists as an element of $L_2(\Omega,\mathscr F_t,\mathbb P;D((-A)^\rho))$ if the integral
\begin{equation}
\label{eq:stochIntegral}
\int_0^t\mathbb E\big\|e^{(t-s)A}B(u(s))\big\|_{\hs(\ell_2;D((-A)^\rho))}^2\md s
\end{equation}
is finite, where $\hs(\ell_2;D((-A)^\rho))$ denotes the space of Hilbert-Schmidt operators from $\ell_2$ to $D((-A)^\rho)$. The integrand in \eqref{eq:stochIntegral} can be estimated from above by
\begin{equation*}
\operatorname{Tr}(-A)^{-\alpha}\big\|(-A)^{\beta+\alpha/2} e^{(t-s)A}\big\|_{\mathcal L(D((-A)^\rho))}^2\mathbb E\big\|(-A)^{-\beta}B(u(s))\big\|_{\mathcal L(\ell_2;D((-A)^\rho))}^2,
\end{equation*}
and we have
\begin{equation*}
\big\|(-A)^{\beta+\alpha/2} e^{(t-s)A}\big\|_{\mathcal L(D((-A)^\rho))}^2\leq C (t-s)^{-(2\beta+\alpha)}
\end{equation*}
for $2\beta+\alpha\in[0,1)$. For $2\beta+\alpha<0$ we have $(-\lambda_j)^{2\beta+\alpha}\leq(-\lambda_1)^{2\beta+\alpha}$ for all $j\in\N_0$, which yields
\begin{equation*}
\big\|(-A)^{\beta+\alpha/2} e^{(t-s)A}\big\|_{\mathcal L(D((-A)^\rho))}^2\leq C.
\end{equation*}
Moreover, by the global Lipschitz continuity of the mapping $B:D((-A)^\rho)\to\mathcal L(\ell_2;D((-A)^{\rho-\beta}))$,
\begin{equation*}
\mathbb E\big\|(-A)^{-\beta}B(u(s))\big\|_{\mathcal L(\ell_2;D((-A)^\rho))}^2\leq C\Big(1+\sup_{r\in[0,T]}\mathbb E\|u(r)\|_{D((-A)^\rho)}^2\Big).
\end{equation*}
Thus, the stochastic integral in \eqref{eq:mildSolution} is well-defined.

\noindent\textit{(ii)} 
For the case $\rho=0$ existence and uniqueness of a mild solution to Eq.~\eqref{SPDE} has been stated in \cite[Proposition~3.1]{bib:Printems2001}. The proof consists of a modification of the proof of Theorem~7.4 in \cite{bib:DPZ92}---a contraction argument in $L_\infty([0,T];L_2(\Omega;\U))$.
For the general case $\rho\geq0$ existence and uniqueness can be proved analogously, see \cite[Theorem~5.1]{bib:JeKl11}.
Alternatively, the case $\rho > 0$ can be traced back to the case $\rho=0$ as described in the proof of Proposition~\ref{prop:SPDEexistence} below.
\end{Remark}

\begin{Proposition}\label{prop:SPDEexistence}
Let Assumptions~\ref{AssumptionASPDE}, \ref{AssumptionfBSPDE} and \ref{Assumptionu0BSPDE} be fulfilled.
Then, Eq.~\eqref{SPDE} has a unique (up to modifications) mild solution in $D((-A)^{\rho})$.
\end{Proposition}

\begin{proof} 
If Assumptions~\ref{AssumptionASPDE}, \ref{AssumptionfBSPDE} and \ref{Assumptionu0BSPDE} are fulfilled for $\rho=0$, Eq.~\eqref{SPDE} fits into the setting of \cite{bib:Printems2001} (the Hilbert space $\U$ is denoted by $H$ there). By Proposition~3.1 therein there exists a unique mild solution $u$ to Eq.~\eqref{SPDE}. Now suppose that Assumptions~\ref{AssumptionASPDE}, \ref{AssumptionfBSPDE} and \ref{Assumptionu0BSPDE} hold for some $\rho>0$. Set
\begin{equation*}
\hat U:=D((-A)^\rho),\quad D(\hat A):=D((-A)^{\rho+1})
\end{equation*}
and consider the unbounded operator $\hat A$ on $\hat U$ given by
\begin{equation*}
\hat A:D(\hat A)\subset\hat U\to\hat U,\;v\mapsto \hat Av:=Av.
\end{equation*}
Note that $\hat A$ fulfills Assumption~\ref{AssumptionASPDE} with $A$, $D(A)$ and $U$ replaced by $\hat A$, $D(\hat A)$ and $\hat U$, respectively. Defining the spaces $D((-\hat A)^s)$ analogously to the spaces $D((-A)^s)$, we have $D((-A)^{\rho+s})=D((-\hat A)^{s})$, $s\in\R$, so that Assumptions~\ref{AssumptionfBSPDE} and \ref{Assumptionu0BSPDE} can be reformulated with $\rho$, $D((-A)^\rho)$, $D((-A)^{\rho-\sigma})$ and $D((-A)^{\rho-\beta})$ replaced by $\hat\rho:=0$, $D((-\hat A)^{\hat\rho})$, $D((-\hat A)^{\hat\rho-\sigma})$ and $D((-\hat A)^{\hat\rho-\beta})$, respectively. Thus, the equation
\begin{equation}\label{shiftedSPDE}
\md u(t) = \big(\hat A u(t) + f(u(t))\big) \md t + B(u(t)) \md W(t), \quad u(0) = u_0,
\end{equation}
fits into the setting of \cite{bib:Printems2001} (now $\hat\U$ corresponds to the space $H$ there), so that, by \cite[Proposition~3.1]{bib:Printems2001}, there exists a unique mild solution $u$ to Eq.~\eqref{shiftedSPDE}. Since the operators $e^{tA}\in\mathcal L(\U)$ and $e^{t\hat A}\in\mathcal L(\hat\U)$ coincide on $\hat\U\subset\U$, it is clear that any mild solution to Eq.~\eqref{shiftedSPDE} is a mild solution to Eq.~\eqref{SPDE} and vice versa.
\end{proof}

\begin{Remark} 
If the initial condition $u_0$ belongs to $L_p(\Omega,\mathscr F_0,\mathbb P;D((-A)^\rho))$\linebreak$\subset L_2(\Omega,\mathscr F_0,\mathbb P;D((-A)^\rho))$ for some $p>2$, the solution $u$ even satisfies\linebreak $\sup_{t\in[0,T]}\mathbb E\|u(t)\|_{D((-A)^\rho)}^p <\infty$. This is a consequence of the Burkholder-Davis-Gundy inequality, cf. \cite[Theorem~7.4]{bib:DPZ92} or \cite[Proposition~3.1]{bib:Printems2001}. Analogous improvements are valid for the estimates in
Propositions~\ref{errorEulerSPDE} and \ref{prop:OverallErrorEstimateSPDE} below.
\end{Remark}

We finish this section with concrete examples for stochastic PDEs that fit into our setting. 

\begin{Example}\label{ConcreteExampleSPDE1}
Let $\domain$ be an open and bounded subset of $\R^d$, $U:=L_2(\domain)$, and let $A=\Delta_\domain^D$ be the Dirichlet-Laplacian on $\domain$ as described in Example~\ref{example:Operator}. We consider examples for stochastic PDEs in dimension $d=1$ and $d\geq2$.

First, let $\domain\subset\R^1$ be one-dimensional and consider the problem
\begin{equation}\label{eq:ConcreteExampleSPDE1}
\left.
\begin{alignedat}{2}
\md u(t,x)&=\Delta_xu(t,x)&&\md t + g(u(t,x))\md t + h(u(t,x))\,\md W_1(t,x),\\
&&&(t,x)\in[0,T]\times\domain,\\
u(t,x)&=0, &&(t,x)\in[0,T]\times\partial \domain,\\
u(0,x)&=u_0(x),&&x\in\domain,
\end{alignedat}
\quad\right\}
\end{equation}
where $u_0\in L_2(\domain)$, $g:\R\to\R$ and $h:\R\to\R$ are globally Lipschitz continuous, and $W_1=(W_1(t))_{t\in[0,T]}$ is a Wiener process (with respect to a normal filtration on a complete probability space) whose Cameron--Martin space is some space of functions on $\domain$ that is continuously embedded in $L_\infty(\domain)$, e.g., $W_1$ is a Wiener process with Cameron--Martin space $H^s(\domain)$ for some $s>1/2$. Let $(\psi_k)_{k\in\N}$ be an arbitrary orthonormal basis of the Cameron--Martin space of $W_1$ and define $f$ and $B$ as the Nemytskii type operators
\begin{equation}\label{eq2:ConcreteExampleSPDE1}
\begin{alignedat}{2}
f(v)(x)&:=g(v(x)),&&v\in L_2(\domain),\;x\in\domain,\\
\big(B(v){\bf a}\big)(x)&:=h(v(x))\sum_{k\in\N}a_k\psi_k(x),&\;\;& v\in L_2(\domain),\;{\bf a}=(a_k)_{k\in\N}\in\ell_2,\;x\in\domain.
\end{alignedat}
\end{equation}
Then, Eq.~\eqref{SPDE} is an abstract version of problem \eqref{eq:ConcreteExampleSPDE1}, and the mappings $f$ and $B$ are globally Lipschitz continuous (and thus linearly growing) from $D((-A)^0)=L_2(\domain)$ to $L_2(\domain)$ and from $D((-A)^0)$ to $\mathcal L(\ell_2;L_2(\domain))$, respectively. For $B$ this follows from the estimate
\begin{align*}
\|B(v_1){\bf a}-B(v_2){\bf a}\|_{L_2(\domain)}
&=\Big\|\big(h(v_1)-h(v_2)\big)\sum_{k\in\N}a_k\psi_k\Big\|_{L_2(\domain)}\\
&\leq\|h(v_1)-h(v_2)\|_{L_2(\domain)}\Big\|\sum_{k\in\N}a_k\psi_k\Big\|_{L_\infty(\domain)}\\
&\leq C\|v_1-v_2\|_{L_2(\domain)}\|{\bf a}\|_{\ell_2},
\end{align*}
where the last step is due to the Lipschitz property of $h$ and the assumption that the Cameron--Martin space of $W_1$ is continuously embedded in $L_\infty(\domain)$.
It follows that Assumptions~\ref{AssumptionASPDE}, \ref{AssumptionfBSPDE} and \ref{Assumptionu0BSPDE} are fulfilled for $1/2<\alpha<1$ (compare Example \ref{example:Operator}) and $\rho=\sigma=\beta=0$.

Now let $\domain\subset\R^d$ be $d$-dimensional, $d\geq2$,
and consider the problem \eqref{eq:ConcreteExampleSPDE1} where $u_0\in L_2(\domain)$, $g:\R\to\R$ is globally Lipschitz continuous, $h:\R\to\R$ is constant (additive noise), and $W_1=(W_1(t))_{t\in[0,T]}$ is a Wiener process whose Cameron--Martin space is some space of functions on $\domain$ that is  continuously embedded in
$D((-A)^{-\beta})$ for some $\beta<1/2-d/4$.
One easily sees that the mappings $f$ and $B$, defined as in  \eqref{eq2:ConcreteExampleSPDE1}, are globally Lipschitz continuous (and thus linearly growing) from $D((-A)^0)=L_2(\domain)$ to $L_2(\domain)$ and from $D((-A)^0)$ to $\mathcal L(\ell_2;D((-A)^{-\beta}))$, respectively.
It follows that Assumptions~\ref{AssumptionASPDE}, \ref{AssumptionfBSPDE} and \ref{Assumptionu0BSPDE} are fulfilled for
$\beta<1/2-d/4$, $d/2<\alpha<1-2\beta$, and $\rho=\sigma=0$. Alternatively, we could assume $h$ to be
sufficiently smooth
and replace $h(u(t,x))$ in problem \eqref{eq:ConcreteExampleSPDE1} by, e.g.,
$h\big(\int_\domain k(x,y)u(t,y)\md y\big)$ with a sufficiently smooth kernel $k:\domain\times\domain\to\R$.
\end{Example}

\begin{Example}\label{ConcreteExampleSPDE2}
As in Examples \ref{example:Operator} and \ref{ConcreteExampleSPDE1}, let $A=\Delta_\domain^D$ be the Dirichlet-Laplacian on an open and bounded domain $\domain\subset\R^d$. From the numerical point of view, we are especially interested in stochastic PDEs of type \eqref{SPDE} with $\rho=1/2$. In this case the solution process takes values in the space $D((-A)^{1/2})=H^1_0(\domain)$, and, as we will see later in Proposition~\ref{errorEulerSPDE} and Theorem~\ref{thm:central-theorem-in-stochastic-section}, we obtain estimates for the approximation error in terms of the energy norm
\begin{equation*}
\|v\|_{D((-\Delta_\domain^D)^{1/2})}=\langle\nabla v,\nabla v\rangle_{L_2(\domain)}^{1/2},\qquad v\in H^1_0(\domain).
\end{equation*}
The energy norm is crucial because error estimates for numerical solvers of elliptic problems (which we want to apply in each time step) are usually expressed in terms of this norm,
compare~\cite[Section~4]{bib:CDDFKLRRS2014}, where adaptive wavelet solvers with optimal convergence rates are considered.

First, let $\domain\subset\R^1$ be one-dimensional, and consider the problem
\eqref{eq:ConcreteExampleSPDE1} where $u_0\in H^1_0(\domain)$, $g:\R\to\R$ is globally Lipschitz continuous, $h:\R\to\R$ is linear or constant, and $W_1=(W_1(t))_{t\in[0,T]}$ is a Wiener process whose Cameron--Martin space is some space of functions on $\domain$ that is continuously embedded in $D((-A)^{1/2-\beta})$ for some nonnegative  $\beta<1/4$, so that $W_1$ takes values in a bigger Hilbert space, say, in $D((-A)^{-1/4})$. (The embedding $D((-A)^{1/2-\beta})\hookrightarrow D((-A)^{-1/4})$ is Hilbert--Schmidt since \eqref{eq:trace} is fulfilled for $\alpha>1/2$, compare Example \ref{example:Operator}.)
Take an arbitrary orthonormal basis $(\psi_k)_{k\in\N}$ of the Cameron--Martin space of $W_1$, and define $f(v)$ and $B(v)$ for $v\in H^1_0(\domain)$ analogously to \eqref{eq2:ConcreteExampleSPDE1}.
Then, Eq.~\eqref{SPDE} is an abstract version of problem \eqref{eq:ConcreteExampleSPDE1}, and the mappings $f$ and $B$ are globally Lipschitz continuous (and thus linearly growing) from $D((-A)^{1/2})=H^1_0(\domain)$ to $D((-A)^0)=L_2(\domain)$ and from $D((-A)^{1/2})$ to $\mathcal L(\ell_2;D((-A)^{1/2-\beta}))$, respectively. The mapping properties of $B$ follow from the fact that the Cameron--Martin space of $W_1$ is continuously embedded in $D((-A)^{1/2-\beta})$ and the inequality \[\|vw\|_{D((-A)^{1/2-\beta})}\leq C\|v\|_{H^1_0(\domain)}\|w\|_{{D((-A)^{1/2-\beta})}}.\] 
The latter is due to the inequalities $\|vw\|_{L_2(\domain)}\leq \|v\|_{H^1_0(\domain)}\|w\|_{L_2(\domain)}$ together with $\|vw\|_{H^1_0(\domain)}\leq C \|v\|_{H^1_0(\domain)}\|w\|_{H^1_0(\domain)}$ (a consequence of the Sobolev embedding $H^1(\domain)\hookrightarrow L_\infty(\domain)$ in dimension $1$) and interpolation since $$D((-A)^{1/2-\beta})=[L_2(\domain),D((-A)^{1/2})]_{1-2\beta}.$$ Thus, Assumptions~\ref{AssumptionASPDE}, \ref{AssumptionfBSPDE} and \ref{Assumptionu0BSPDE} are fulfilled for $\rho=\sigma=1/2$, $0\leq\beta <1/4$ and $1/2<\alpha<1-2\beta$.

Now let $\domain\subset\R^d$ be $d$-dimensional and consider problem \eqref{eq:ConcreteExampleSPDE1} where $u_0\in H^1_0(\domain)$, $g:\R\to\R$ is globally Lipschitz continuous, $h:\R\to\R$ is constant, and $W_1=(W_1(t))_{t\in[0,T]}$ is a Wiener process whose Cameron--Martin space is continuously embedded in $D((-A)^{1/2-\beta})$ for some $\beta<1/2-d/4$. Then, the mappings  $f$ and $B$, defined analogously to the one dimensional case, are globally Lipschitz continuous (and thus linearly growing) from $D((-A)^{1/2})=H^1_0(\domain)$ to $D((-A)^0)=L_2(\domain)$ and from $D((-A)^{1/2})$ to $\mathcal L(\ell_2;D((-A)^{1/2-\beta}))$ respectively. It follows that Assumptions~\ref{AssumptionASPDE}, \ref{AssumptionfBSPDE} and \ref{Assumptionu0BSPDE} are fulfilled for $\rho=\sigma=1/2$, $\beta<1/2-d/4$  and $1<\alpha<1-2\beta$. As in Example~\ref{ConcreteExampleSPDE1} we could alternatively assume $h:\R\to\R$ to be sufficiently smooth and replace $h(u(t,x))$ in problem \eqref{eq:ConcreteExampleSPDE1} by, e.g., 
$h\big(\int_\domain k(x,y)u(t,y)\md y\big)$ with a sufficiently smooth kernel $k:\domain\times\domain\to\R$.
\end{Example}

\begin{Remark}\label{rem:Setting}
The reader familiar with SPDEs of the form \eqref{SPDE} might wonder about the rather restrictive conditions in the examples above, especially on the noise terms therein.
These restrictions are due to the fact that we basically adopt the setting from \cite{bib:Printems2001}.
This allows us to focus on our main goal, i.e., the analysis of the error of the inexact counterpart of the Euler scheme~\eqref{eq:EulerSPDE}, without spending too much time on explaining details regarding the underlying setting.
However, it is worth mentioning that much more general equations of the type~\eqref{SPDE} have been considered in the literature, see, e.g., the recent results concerning the maximal $L_p$-regularity of SPDEs in \cite{bib:NeeVerWei2012}. 
Also, the convergence of the linearly implicit Euler scheme has been considered under weaker assumptions, see, e.g., \cite{bib:Cox2012, bib:CVN2012}.
\end{Remark}

\section{Exact Euler scheme}
\label{sec:semi-discretization-in-time}

In this section we use the linearly implicit Euler scheme to obtain a semidiscretization of Eq.~\eqref{SPDE} in time.
We present a corresponding convergence result as a slight modification of \cite[Theorem~3.2]{bib:Printems2001}.
Since no spatial discretization is involved, we speak of the \emph{exact} Euler scheme in contrast to the inexact perturbed scheme considered in the forthcoming section.
From now on, let Assumptions~\ref{AssumptionASPDE}, \ref{AssumptionfBSPDE} and \ref{Assumptionu0BSPDE} be fulfilled.

For $K\in\mathbb N$ and $\tau:=T/K$ we consider discretizations $(u_k)_{k=0}^K$ given by the linearly implicit Euler scheme \eqref{eq:EulerSPDE}, i.e.,
\begin{equation*}
u_{k+1}:=(I-\tau A)^{-1}\big(u_k+\tau f(u_k)+\sqrt\tau B(u_k)\chi_{k}\big), \qquad
k=0,\ldots,K-1.
\end{equation*}
We use the abbreviation
\begin{equation*}
\chi_{k}:=\chi_{k}^K:=\frac 1{\sqrt\tau}\left(W\big(t_{k+1}^K\big)-W\big(t_k^K\big)\right)
\end{equation*}
with 
\begin{equation*}
t_k:=\tau k,\quad k=0,\ldots,K.
\end{equation*}
Note that each $\chi_{k}$, $k=0,\ldots,K-1$, is an $\mathscr F_{t_{k+1}}$-measurable Gaussian white noise on $\ell_2$, i.e., a linear isometry from $\ell_2$ to $L_2(\Omega,\mathscr F_{t_{k+1}},\mathbb P)$ such that for each $\mathbf{a}\in\ell_2$ the real valued random variable $\chi_k(\mathbf{a})$ is centered Gaussian with variance $\|\mathbf a\|_{\ell_2}^2$. 
Moreover, for each $k=0,\ldots,K-1$, the sub-$\sigma$-field of $\mathscr F$ generated by $\{\chi_k(\mathbf a):\mathbf a\in\ell_2\}$ is independent of $\mathscr F_{t_{k}}$.

We explain in which way the scheme \eqref{eq:EulerSPDE} has to be understood.
Let $G$ be a separable real Hilbert space such that $D((-A)^{\rho-\beta})$ is embedded into $G$ via a Hilbert--Schmidt embedding.  Then, for all $k=0,\ldots,K-1$ and for all $\mathscr F_{t_{k}}$-measurable, $D((-A)^\rho)$-valued,  square integrable random variables $v\in L_2(\Omega,\mathscr F_{t_{k}},\mathbb P;D((-A)^\rho))$, the term $B(v)\chi_{k}$ can be interpreted as an $\mathscr F_{t_{k+1}}$-measurable, square integrable, $G$-valued random variable in the sense
\begin{equation}\label{eq:noiseExpansion}
B(v)\chi_{k}:=L_2(\Omega,\mathscr F_{t_{k+1}},\mathbb P;G)\text{-}\lim_{J\to\infty}\sum_{j=1}^J\chi_{k}(\mathbf b_j)B(v)\mathbf b_j,
\end{equation}
where $(\mathbf b_j)_{j\in\mathbb N}$ is an orthonormal basis of $\ell_2$.  
This definition does not depend on the specific choice of the orthonormal basis $(\mathbf b_j)_{j\in\mathbb N}$. Note that the stochastic independence of  $\{\chi_{k}(\mathbf a):\mathbf a\in\ell_2\}$ and $\mathscr F_{t_k}$ is important at this point. We have
\begin{equation}\label{eq:isometry-wn}
 \mathbb{E} \lVert B(v)\chi_{k}\rVert_G^2 = \mathbb E\lVert B(v)\rVert_{\hs(\ell_2;G)}^2,
\end{equation}
the last term being finite due to the Lipschitz continuity of $B$ by Assumption~\ref{AssumptionfBSPDE} (see also Remark~\ref{rem:AssumptionBSPDE}) and the fact that the embedding \begin{equation*}D((-A)^{\rho-\beta})\hookrightarrow G\end{equation*} is Hilbert--Schmidt. 
Let us explicitly choose the space $G$ in such a way that the terms $u_k+\tau f(u_k)+ \sqrt{\tau}B(u_k)\chi_k$ on the right hand side of \eqref{eq:EulerSPDE} can be considered as a $G$-valued random variable and the application of $(I-\tau A)^{-1}$ to elements of $G$ makes sense.
Our choice of $G$, which we keep throughout this paper, is
\begin{equation}\label{eq:G}
G:=D((-A)^{\rho-\max(0,\sigma,\beta+\alpha/2)}).
\end{equation}
Note that the condition $\text{Tr}(-A)^{-\alpha}<\infty$ in Assumption~\ref{AssumptionASPDE} yields that the embedding $D((-A)^{\rho-\beta})\hookrightarrow D((-A)^{\rho-\beta-\alpha/2})$ is Hilbert--Schmidt, and the embedding $D((-A)^{\rho-\beta-\alpha/2})\hookrightarrow
D((-A)^{\rho-\max(0,\sigma,\beta+\alpha/2)})$
is clearly continuous.
Thus, we have indeed a Hilbert--Schmidt embedding $D((-A)^{\rho-\beta})\hookrightarrow G$.
For all $k=0,\ldots,K-1$ and $v\in L_2(\Omega,\mathscr F_{t_k},\mathbb P;D((-A)^\rho))$ we consider the term $B(v)\chi_{k}$ as an element in the space
\begin{equation*}
L_2(\Omega,\mathscr F_{t_{k+1}},\mathbb P;G)=L_2\big(\Omega,\mathscr F_{t_{k+1}},\mathbb P;D((-A)^{\rho-\max(0,\sigma,\beta+\alpha/2)})\big).
\end{equation*}
Next, due to the Lipschitz continuity of $f$ by 
Assumption~\ref{AssumptionfBSPDE} (see also Remark~\ref{rem:AssumptionBSPDE}), we also know that for all $v\in L_2(\Omega,\mathscr F_{t_k},\mathbb P;D((-A)^\rho))$ the term $f(v)$ is an element in $L_2(\Omega,\mathscr F_{t_k},\mathbb P;G)$.
Finally, as a consequence of Lemma~\ref{lem:estimatesResolvent} below and the fact that $\max(0,\sigma,\beta+\alpha/2)<\max(\sigma,1/2)<1$ due to \eqref{eq:rsbeta}, the operator
\begin{equation*}
(I-\tau A)^{-1} : G\to D((-A)^\rho) 
\end{equation*}
is continuous. It follows that the discretizations $(u_k)_{k=0}^K$ are uniquely determined by \eqref{eq:EulerSPDE} and for all $k=0,\ldots,K$ we have
\begin{equation*} 
u_k\in L_2(\Omega,\mathscr F_{t_k},\mathbb P; D((-A)^\rho)).
\end{equation*}

Now we prove the lemma we just used to show the boundedness of the resolvents  $(I-\tau A)^{-1}$ of $A$ in the right spaces. It will also be employed to prove Proposition~\ref{prop:LipschitzBoundSPDE} in the next section.
\begin{Lemma}\label{lem:estimatesResolvent}
Let $\tau>0$ and $r\in\mathbb{R}$. The operator $I-\tau A$ is a homeomorphism from $D((-A)^r)$ to $D((-A)^{r-1})$. For $n\in\mathbb{N}$ we have the following operator norm estimates for $(I-\tau A)^{-n}$, considered as an operator from $D((-A)^{r-s})$ to $D((-A)^r)$, $s\leq1$:
\begin{equation*}
\lVert(I-\tau A)^{-n}\rVert_{\mathcal L(D((-A)^{r-s}),D((-A)^r))}\leq
\begin{cases}
\displaystyle s^s\Big(1-\frac{s}{n}\Big)^{(n-s)} (n\tau)^{-s} &\!\!\! \mbox{, } 0<s\leq1\\[.9em]
(-\lambda_1)^{s}(1-\tau\lambda_1)^{-n} &\!\!\! \mbox{, } s\leq 0.
\end{cases}
\end{equation*}
\end{Lemma}

\begin{proof}
The bijectivity of $I-\tau A:D((-A)^r)\to D((-A)^{r-1})$ is almost obvious. Its proof is left to the reader. The bicontinuity follows from the continuity of the inverse as shown below (case $s=1$) and the bounded inverse theorem.
Concerning the operator norm estimates,
we use Parseval's identity and the spectral properties of $A$
to obtain
\begin{align*}
\sup_{\lVert v\rVert_{D((-A)^{r-s})}=1}&\lVert(I-\tau A)^{-n}v\rVert_{D((-A)^{r})}^2\\
&=\sup_{\lVert w\rVert_U=1}\lVert(I-\tau A)^{-n}(-A)^{s-r}w\rVert_{D((-A)^{r})}^2\\
&=\sup_{\lVert w\rVert_U=1}\lVert(-A)^r (I-\tau
A)^{-n}(-A)^{s-r}w\rVert_{U}^2\\
&=\sup_{\lVert w\rVert_U=1}\sum_{k\in\mathbb N}\left|(-\lambda_k)^{s}(1-\tau\lambda_k)^{-n}
\big\langle w,e_k\big\rangle_{U}\right|^2.
\end{align*}
If $s\leq 0$, the last expression is 
equal to $(-\lambda_1)^{2s}(1-\tau\lambda_1)^{-2n}$. If $0<s\leq1$, an upper bound is
given by the square of
\begin{equation*}
\sup_{x>0}\frac{x^s}{(1+\tau x)^n}=s^s\Big(1-\frac{s}{n}\Big)^{(n-s)} (n\tau)^{-s}.\qedhere
\end{equation*}
\end{proof}

After these preparations we present an extension of the error estimate in~\cite{bib:Printems2001}.
\begin{Proposition}\label{errorEulerSPDE}
Let Assumptions~\ref{AssumptionASPDE}, \ref{AssumptionfBSPDE} and \ref{Assumptionu0BSPDE} be fulfilled. Let $(u_k)_{k=0}^{K}$ be the time discretization of the mild solution $(u(t))_{t\in[0,T]}$ in $D((-A)^\rho)$ to Eq.~\eqref{SPDE}, given by the scheme \eqref{eq:EulerSPDE}. Then, for every
\begin{equation*}
\delta<\min(1-\sigma,(1-\alpha)/2-\beta),
\end{equation*}
we have for all $1\leq k\leq K$
\begin{equation*}
\Big(\mathbb E\lVert u(t_k)-u_k\rVert_{D((-A)^\rho)}^2\Big)^{1/2}\leq C \left(\tau^\delta + \frac 1k\Big(\mathbb E\lVert u_0\lVert_{D((-A)^\rho)}^2\Big)^{1/2}\right),
\end{equation*}
where the constant $C>0$ depends only on $\delta$, $A$, $B$, $f$, $\alpha$, $\beta$, $\sigma$ and $T$.
\end{Proposition}

\begin{proof}
We argue as in the proof of Proposition~\ref{prop:SPDEexistence} and consider the equation
\begin{equation}\label{shiftedSPDE2}
\md u(t) = \big(\hat A u(t) + f(u(t))\big) \md t + B(u(t)) \md W(t), \quad u(0) = u_0,
\end{equation}
where the operator $\hat A:\;D(\hat A)\subset \hat U\to\hat U$ is defined by $\hat U:=D((-A)^\rho)$, $D(\hat A):=D((-A)^{\rho+1})$, and $\hat Av:=Av$, $v\in D(\hat A)$.
Eq.~\eqref{shiftedSPDE2} fits into the setting of \cite{bib:Printems2001}, and its mild solution $u=(u(t))_{t\in[0,T]}$ coincides with the mild solution to Eq.~\eqref{SPDE}, compare the proof of Proposition~\ref{prop:SPDEexistence}.
For $K\in\mathbb N$ let $(\hat u_k)_{k=0}^K$ be given by the linearly implicit Euler scheme
\begin{equation*}
\begin{alignedat}{2}
\hat u_0&=u_0,&\\
\hat u_{k+1}&=(I-\tau \hat A)^{-1}&\big(\hat u_k+\tau f(\hat u_k)+\sqrt\tau B(\hat u_k)\chi_{k}\big), \quad k=0,\ldots,K-1.
\end{alignedat}
\end{equation*}
By Theorem~3.2 in \cite{bib:Printems2001} we have, for all $\delta<\min(1-\sigma,(1-\alpha)/2-\beta)$,
\begin{equation*}
\Big(\mathbb E\lVert u(t_k)-\hat u_k\rVert_{\hat U}^2\Big)^{1/2}\leq C \left(\tau^\delta + \frac 1k\Big(\mathbb E\lVert u_0\rVert_{\hat U}^2\Big)^{1/2}\right),
\end{equation*}
$1\leq k\leq K$. The proof in \cite{bib:Printems2001} reveals that the constant $C>0$ depends only on $\delta$, $\hat A$, $B$, $f$, $\alpha$, $\beta$, $\sigma$ and $T$. The assertion of Proposition~\ref{errorEulerSPDE} follows from the fact that the natural extensions and restrictions of the operators $(I-\tau\hat A)^{-1}$ and $(I-\tau A)^{-1}$  to the spaces $D((-\hat A)^s)=D((-A)^{s+\rho})$, $s\in\R$, coincide, so that $\hat u_k=u_k$ for all $0\leq k\leq K$, $K\in\N$.
\end{proof}

\begin{Remark}\label{rem:ErrorEulerSPDE}
\textit{(i)} If $k\geq K^\delta$ ($\delta>0$), then $1/k\leq T^{-\delta}\tau^\delta$, and we obtain
\begin{equation}\label{eq:RemarkErrorEulerSPDE}
\Big(\mathbb E\lVert u(t_k)-u_k\rVert_{D((-A)^\rho)}^2\Big)^{1/2}\leq C \, \tau^\delta
\end{equation}
with a constant $C>0$ that depends only on $\delta$, $u_0$, $A$, $B$, $f$, $\alpha$, $\beta$, $\sigma$ and $T$. Since $\delta$ is always smaller than $1$, it follows in particular that \eqref{eq:RemarkErrorEulerSPDE} holds for $k=K$, i.e., 
\begin{equation*}
\Big(\mathbb E\lVert u(T)-u_K\rVert_{D((-A)^\rho)}^2\Big)^{1/2}\leq C \, \tau^\delta.
\end{equation*}

\noindent\textit{(ii)} The proof of Proposition~\ref{errorEulerSPDE} is based on an application of Theorem~3.2 in \cite{bib:Printems2001} to Eq.~\eqref{shiftedSPDE}. The  reader might have observed that therein the parameter $s$, which corresponds to our parameter $\sigma$, is assumed to be positive. However, a closer look at the estimates in the proof of Theorem~3.2 in \cite{bib:Printems2001} reveals that the result can be extended to negative values of $\sigma$ and $s$, respectively. Alternatively, one can argue that if $\sigma\leq0$ then $D((-\hat A)^{-\sigma})$ is  continuously embedded into, say, $D((-\hat A)^{-1/2})$, so that Eq.~\eqref{shiftedSPDE} fits into the setting of \cite{bib:Printems2001} if $f$ is considered as a mapping from $\hat U=D((-\hat A)^0)$ to $D((-\hat A)^{-1/2})$. We refer to \cite{bib:CVN2012} where the Euler scheme for stochastic evolution equations is considered in a more general setting than in \cite{bib:Printems2001}.
\end{Remark}

\section{Error control for the inexact scheme}\label{sec:discretization-in-time-and-space}

So far we have verified the existence and uniqueness of a mild solution to Eq.~\eqref{SPDE} as well as the convergence of the exactly evaluated linearly implicit Euler scheme \eqref{eq:EulerSPDE} with rate $\delta<\min(1-\sigma,(1-\alpha)/2-\beta)$.
We now turn to the corresponding inexact scheme. 
We assume that we have at hand a `black box' numerical scheme, which for any element $w\in L_2(\Omega,\mathscr F_{t_k},\mathbb{P};D((-A)^\rho)$ approximates 
\[
v:=(I-\tau A)^{-1} \big(w + f(w) + B(w)\chi_k \big)
\]
with a prescribed tolerance $\varepsilon>0$, the error being measured with respect to the   $L_2(\Omega,\mathscr F_{t_{k+1}},\mathbb{P};D((-A)^\rho)$-norm.
What we have in mind are $\omega$-wise applications of some deterministic solver for elliptic equations, e.g., an optimal adaptive wavelet solver as developed in \cite{bib:CDD01}, combined with proper evaluations of the nonlinearities $f$ and $B$, see, e.g., \cite{bib:Kap2011, bib:DahSchXu2000, bib:CohDahDeV2000}, and an adequate truncation of the noise $B(w)\chi_k$.
We start with the initial condition $u_0$ and in each time step, we apply this `black box' method to the approximation we have obtained in the step before. 
Our main goal is to prove that the tolerances in the different time steps can be chosen in such a way that the inexact scheme achieves the same convergence rate (in time direction) as its exact counterpart (Theorem~\ref{thm:central-theorem-in-stochastic-section}). 
To this end we also analyze the error propagation of the inexact scheme (Proposition~\ref{prop:OverallErrorEstimateSPDE}).

Our strategy relies on the ideas presented in \cite{bib:CDDFKLRRS2014}. Therein, Rothe's method for deterministic parabolic equations is analyzed by putting it into an abstract framework, cf.\ in particular \cite[Section~2]{bib:CDDFKLRRS2014}.
We proceed similarly. Therefore, we first of all reformulate the exact linearly implicit Euler scheme~\eqref{eq:EulerSPDE} in the following way.
We set 
\begin{equation}\label{eq:SPDEinterpretation1}
\left.
\begin{alignedat}{2}
\H_k&:=L_2(\Omega,\mathscr F_{t_k},\mathbb P;D((-A)^\rho)),&& k=0,\ldots,K,\\
\G_k&:=L_2(\Omega,\mathscr F_{t_k},\mathbb P;G),&& k=1,\ldots,K,\\
%=L_2\big(\Omega,\mathscr F,\mathbb P;D((-A)^{r-(s\vee(\beta+\alpha/2))})\big)
R_{\tau,k}&:\H_{k}\to\G_{k+1}\\
&\quad v\mapsto R_{\tau,k}(v):= v+\tau f(v)+\sqrt{\tau} B(v)\chi_{k},&\quad& k=0,\ldots,K-1,\\
L_{\tau}^{-1}&:\G_k\to\H_k\\
&\quad v\mapsto L_{\tau}^{-1} v:=(I-\tau A)^{-1}v,&& k=1,\ldots,K.
\end{alignedat}
\right\}
\end{equation}
Recall that $G=D((-A)^{\rho-\max(0,\sigma,\beta+\alpha/2)})$ has been introduced in \eqref{eq:G}.
With these definitions at hand, we can rewrite the scheme~\eqref{eq:EulerSPDE} as 
\begin{equation}\label{eq:abstractEulerSPDE}
u_{k+1}:=L_{\tau}^{-1}R_{\tau,k}(u_k),\quad k=0,\ldots,K-1.
\end{equation}

\begin{Remark}
Without additional assumptions on $B$ or a truncation of the noise expansion \eqref{eq:noiseExpansion},
the operator $R_{\tau,k}$ cannot easily be traced back to a family of operators
\begin{equation*}
R_{\tau,k,\omega}:D((-A)^\rho)\to G, \; \omega\in\Omega,
\end{equation*}
in the sense that for $v\in\mathcal H_{k}=L_2(\Omega,\mathscr F_{t_{k}},\mathbb P;D((-A)^\rho))$ the image $R_{\tau,k}(v)$ is determined by
\begin{equation}\label{eq:omegawiseR}
(R_{\tau,k}(v))(\omega)=R_{\tau,k,\omega}(v(\omega))\quad\text{  for } \mathbb P\text{-almost all }\omega\in\Omega.
\end{equation}
However,
this is possible, for instance, if for all $v\in D((-A)^\rho)$ the operator $B(v):\ell_2\to D((-A)^{\rho-\beta})$ has a continuous extension $B(v):U_0\to D((-A)^{\rho-\beta}$) to a bigger Hilbert space $U_0$ such that $\ell_2$ is embedded into $U_0$ via a Hilbert--Schmidt embedding.
Another instance where a representation of the form \eqref{eq:omegawiseR} is possible is the case where the mapping $B:D((-A)^\rho)\to\mathcal L(\ell_2;D((-A)^{\rho-\beta}))$ is constant, i.e., the case of additive noise. We take a closer look at the latter case, writing $B\in\mathcal L(\ell_2;D((-A)^{\rho-\beta}))$ for short. We fix a version of each of the $\mathbb P$-almost surely determined, $G$-valued random variables $B\chi_{k}=B\chi_{k}^K$, $k=0,\ldots,K-1$, $K\in\mathbb N$, and set
\begin{equation*}
R_{\tau,k,\omega}(v):=v+f(v)
+(B\chi_{k})(\omega), \qquad \omega\in\Omega,\; v\in D((-A)^\rho).
\end{equation*}
It is clear that \eqref{eq:omegawiseR} holds for all $v\in L_2(\Omega,\mathscr F_{t_{k}},\mathbb P;D((-A)^\rho))$. Thus, in the case of additive noise,  
by setting
\begin{equation}%\tag{\ref{eq:SPDEinterpretation1}$_\omega$}
\left.
\begin{aligned}
\mathcal H_k&:=D((-A)^\rho)),\\
\mathcal G_k&:=G=D((-A)^{\rho-\max(0,\sigma,\beta+\alpha/2)}),\\
%=L_2\big(\Omega,\mathscr F,\mathbb P;D((-A)^{r-(s\vee(\beta+\alpha/2))})\big)
R_{\tau,k,\omega}&:\mathcal H_k\to\mathcal G_{k+1}\\
&\quad v\mapsto R_{\tau,k,\omega}(v):= v+\tau f(v)+ \sqrt{\tau} (B\chi_{k})(\omega),\\
L_\tau^{-1}&:\mathcal G_k\to\mathcal H_k\\
&\quad v\mapsto L_\tau^{-1} v:=(I-\tau A)^{-1}v,
\end{aligned}\quad
\right\}
\end{equation}
$k=0,\ldots,K-1$,
we can rewrite the scheme \eqref{eq:EulerSPDE} in an $\omega$-wise sense as 
\[
u_{k+1}(\omega):=L_{\tau}^{-1}R_{\tau,k,\omega}(u_k(\omega)),\quad k=0,\ldots,K-1.
\]
\end{Remark}

In this section we are focusing on the inexact counterpart of the scheme \eqref{eq:EulerSPDE}, which we introduce now. 
We assume that we have a numerical scheme which, for any $k=0,\ldots,K-1$, any $w\in\H_k$, and any prescribed tolerance $\epsilon>0$, provides an approximation $[v]_\epsilon$ of
\begin{equation*}
v=L_\tau^{-1}R_{\tau,k}(w),
\end{equation*}
such that
\begin{equation*}
\lVert v-[v]_\epsilon\rVert_{\H_{k+1}}
=\left(\mathbb E\|v-[v]_{\epsilon}\|_{D((-A)^\rho)}^2\right)^{1/2}\leq\epsilon.
\end{equation*}
Given prescribed tolerances $\epsilon_k$, $k=0,\ldots,K-1$, for the spatial approximation errors in each time step, we consider the \emph{inexact linearly implicit Euler scheme}, defined as follows: %\eqref{eq:abstractEulerSPDE}
\begin{equation}\label{eq:abstractEulerSPDE-ie}
\left.
\begin{aligned}
\tilde u_0&:=u_0,\\
\tilde u_{k+1}&:= [L_\tau^{-1}R_{\tau,k}(\tilde{u}_{k})]_{\epsilon_{k}},\quad k=0,\ldots,K-1.
\end{aligned}
\,\right\}
\end{equation}
Note that the errors at each time step accumulate due to the iterated application of the numerical method $[\cdot]_\varepsilon$.

Next we present the main result of this paper. It describes a way how to choose the tolerances in the different time steps so that the inexact scheme~\eqref{eq:abstractEulerSPDE-ie} has the same convergence rate as its exact counterpart \eqref{eq:abstractEulerSPDE}. 

\begin{Theorem}\label{thm:central-theorem-in-stochastic-section}
Let Assumptions
\ref{AssumptionASPDE}, \ref{AssumptionfBSPDE} and \ref{Assumptionu0BSPDE} be fulfilled, $(u(t))_{t\in[0,T]}$ be the unique mild solution in $D((-A)^\rho)$ to Eq.~\eqref{SPDE} and let
\begin{equation*}
\delta<\min(1-\sigma,(1-\alpha)/2-\beta).
\end{equation*}
If one chooses
\[\epsilon_k\leq \tau^{1+\delta}\]
for all $k=0,\ldots,K-1$, $K\in\N$, then the output $\tilde u_K$ of the inexact linearly implicit Euler scheme \eqref{eq:abstractEulerSPDE-ie} converges to $u(T)$ with rate $\delta$, i.e., we have
\begin{equation*}
\left(\mathbb E\|u(T)-\tilde u_K\|_{D((-A)^\rho)}^2\right)^{1/2}
\leq C \tau^\delta
\end{equation*}
with a constant $C$ depending only on $u_0$, $\delta$, $A$, $B$, $f$, $\alpha$, $\beta$, $\sigma$ and $T$.
\end{Theorem}

Our strategy for proving this theorem relies on two auxiliary results, which we prove first. 
We start with Proposition~\ref{prop:OverallErrorEstimateSPDE}, which states that the error propagation of the inexact linearly implicit Euler scheme can be described in terms of the Lipschitz constants $\Lip{j}{k}$ of the operators
\begin{equation}\label{eq:OperatorsExact}
E_{\tau,j,k}=(L_\tau^{-1} R_{\tau,k-1}) \circ (L_\tau^{-1} R_{\tau,k-2}) \circ \dots \circ (L_\tau^{-1} R_{\tau, j}):\H_j\to\H_k,
\end{equation}
$1\leq j\leq k\leq K,\;K\in\N$. 
Afterwards, we prove that these Lipschitz constants are bounded from above uniformly in $1\leq j,k \leq K$, $K\in\N$, see Proposition~\ref{prop:LipschitzBoundSPDE}.
Finally, at the end of this section, we draw the proof of the main Theorem~\ref{thm:central-theorem-in-stochastic-section}.

\begin{Proposition}\label{prop:OverallErrorEstimateSPDE}
Let Assumptions~\ref{AssumptionASPDE}, \ref{AssumptionfBSPDE} and \ref{Assumptionu0BSPDE} be fulfilled. Let $(u(t))_{t\in[0,T]}$ be the unique mild solution in $D((-A)^\rho)$ to Eq.~\eqref{SPDE}.
Let $(\tilde u_k)_{k=0}^{K}$ be the discretization of $(u(t))_{t\in[0,T]}$ in time and space given by the inexact linearly implicit Euler scheme \eqref{eq:abstractEulerSPDE-ie},
where $\epsilon_k$, $k=0,\ldots,K-1$, are prescribed tolerances for the spatial approximation errors in each time step.
Then, for every $1\leq k\leq K$, $K\in\mathbb N$, and for every
\begin{equation*}
\delta<\min(1-\sigma,(1-\alpha)/2-\beta)
\end{equation*}
we have
\begin{align*}
\left(\mathbb E\|u(t_k)-\tilde u_k\|_{D((-A)^\rho)}^2\right)^{1/2}
& \leq
C\bigg(\tau^\delta + \frac 1k  \Big(\mathbb E\|u_0\|_{D((-A)^\rho)}^2\Big)^{1/2}\bigg)\\
&\qquad+\sum_{j=0}^{k-1}\Lip{j+1}{k}\epsilon_j
\end{align*}
with a constant $C$  that depends only on $\delta$, $A$, $B$, $f$, $\alpha$, $\beta$, $\sigma$ and $T$.
\end{Proposition}

\begin{proof}
We argue following the lines of the proof of \cite[Theorem~2.16]{bib:CDDFKLRRS2014}. 
To this end, we introduce the operators
\begin{align*}
\tilde{E}_{\tau,k,k+1}: \H_k & \to \H_{k+1}\\
v & \mapsto \tilde{E}_{\tau,k,k+1}(v) := [L_\tau^{-1} R_{\tau,k}(v)]_{\varepsilon_k}
\end{align*}
for $k=0,\ldots,K-1$, and for $0\leq j\leq k$ we set 
\begin{equation*}
    \tilde{E}_{\tau,j,k} 
    := 
    \begin{cases}
        \tilde{E}_{\tau,k-1,k} \circ \ldots \circ \tilde{E}_{\tau,j,j+1}, & j<k\\
        I, & j=k. 
    \end{cases}
\end{equation*}
Note that, in order to define the operators $E_{\tau,j,k}$ and $\tilde{E}_{\tau,j,k}$ properly, we have to consider the measurability of the elements in their domains and in the corresponding image spaces, cf. our remarks concerning the right understanding of the scheme~\eqref{eq:EulerSPDE} at the beginning of Section~\ref{sec:semi-discretization-in-time}.
Therefore, in contrast to the situation in \cite[Section~2]{bib:CDDFKLRRS2014}, these Hilbert spaces depend on the indexes $j$ and $k$.
Nevertheless, with the right changes, we can argue along the lines of \cite[Theorem~2.16]{bib:CDDFKLRRS2014}:
We rewrite the difference $u_k-\tilde{u}_k$ between the output of the exact and inexact scheme at time $t_k$ as an appropriate telescopic sum, so that by consecutive applications of the triangle inequality in $\H_k$ we obtain:
\begin{align*}
\lVert u(t_k) - \tilde{u}_k \rVert_{\H_k}
&\leq
\lVert u(t_k) - u_k \rVert_{\H_k}
+
\lVert u_k-\tilde{u}_k \rVert_{\H_k}\\
&\leq
\lVert u(t_k) - u_k \rVert_{\H_k}\\
& \qquad +
\sum_{j=0}^{k-1}
\big\lVert E_{\tau,j,k} \tilde{E}_ {\tau,0,j} (u_0) - E_{\tau,j+1,k}\tilde{E}_{\tau,0,j+1}(u_0)\big\rVert_{\H_k}.
\end{align*}
Note that $\Lip jk <\infty$ for all $1\leq j\leq k\leq K$, $K\in\N$, because of Assumption~\ref{AssumptionfBSPDE} on the Lipschitz continuity of the free terms $f$ and $B$ and because of Lemma~\ref{lem:estimatesResolvent} on the boundedness of the resolvents of $A$.
Thus, each term in the sum on the right hand side can be estimated as follows
\begin{align*}
\big\lVert E_{\tau,j,k} \tilde{E}_ {\tau,0,j} (u_0) - & E_{\tau,j+1,k}\tilde{E}_{\tau,0,j+1}(u_0)\big\rVert_{\H_k}\\
&=
\big\lVert E_{\tau,j+1,k}E_{\tau,j,j+1} \tilde{E}_ {\tau,0,j} (u_0) - E_{\tau,j+1,k}\tilde{E}_{\tau,0,j+1}(u_0)\big\rVert_{\H_k}\\
&\leq \Lip{j+1}{k}
\big\lVert E_{\tau,j,j+1} \tilde{E}_ {\tau,0,j} (u_0) - \tilde{E}_{\tau,0,j+1}(u_0)\big\rVert_{\H_{j+1}}.
\end{align*}
Since $\tilde{u}_j=\tilde{E}_{\tau,0,j}(u_0)$, we obtain
\begin{align*}
\big\lVert E_{\tau,j,j+1} \tilde{E}_ {\tau,0,j} (u_0) - & \tilde{E}_{\tau,0,j+1}(u_0)\big\rVert_{\H_{j+1}}\\
&=
\big\lVert E_{\tau,j,j+1} (\tilde{u}_j) - \tilde{E}_{\tau,j,j+1}(\tilde{u}_j)\big\rVert_{\H_{j+1}}
\leq \varepsilon_j.
\end{align*}
Putting these estimates together yields:
\begin{align*}
\lVert u(t_k) - \tilde{u}_k \rVert_{\H_k}
\leq
\lVert u(t_k) - u_k \rVert_{\H_k}
+
\sum_{j=0}^{k-1}
\Lip{j+1}{k} \varepsilon_j.
\end{align*}
The error of the exact Euler scheme at time $t_k$ appearing on the right hand side can be estimated by using Proposition~\ref{errorEulerSPDE}, which yields the assertion.
\end{proof}
 
In order to prove the main Theorem~\ref{thm:central-theorem-in-stochastic-section}, it remains to verify the uniform boundedness of the Lipschitz constants $\Lip jk$ of the operators $E_{\tau,j,k}$ introduced in \eqref{eq:OperatorsExact}. The proof is based on a Gronwall argument.

\begin{Proposition}\label{prop:LipschitzBoundSPDE}
Let Assumptions \ref{AssumptionASPDE}, \ref{AssumptionfBSPDE} and \ref{Assumptionu0BSPDE} be fulfilled.
There exists a finite constant $C>0$, depending only on $A$, $B$, $f$, $\alpha$, $\beta$, $\sigma$ and $T$, such that the Lipschitz constants $\Lip{j}{k}$ of the operators $E_{\tau,j,k}$ introduced in \eqref{eq:OperatorsExact} satisfy:
\begin{equation*}
C_{\tau,j,k}^{\text{\emph{Lip}}}\leq C\quad
\text{ for all }\; 1\leq j\leq k\leq K, \;K\in\mathbb N.
\end{equation*}
\end{Proposition}

\begin{proof}
Fix
$1\leq j \leq k\leq K$ and observe that,
by induction over $k$,
\begin{align*}
E_{\tau,j,k}(v)=\;&L_\tau^{-(k-j)}v\\
&+\sum_{i=0}^{k-j-1}L_\tau^{-(k-j)+i}\Big(\tau f\big(E_{\tau,j,j+i}(v)\big)+\sqrt\tau B\big(E_{\tau,j,j+i}(v)\big)\chi_{j+i}\Big)
\end{align*}
for all $v\in\H_j$, where we set $E_{j,j}=I$. Therefore, for all $v,w\in\H_j$, we have
\begin{align}\label{SPDELipschitzEstimate1}
\| E_{\tau,j,k}&(v)-E_{\tau,j,k}(w)\|_{\H_k} \nonumber\\
&\leq \|L_\tau^{-(k-j)}v-L_\tau^{-(k-j)}w\|_{\H_k} \nonumber\\
&\;\quad+\sum_{i=0}^{k-j-1}\tau\left\|L_\tau^{-(k-j)+i}
\Big(f\big(E_{\tau,j,j+i}(v)\big)-f\big(E_{\tau,j,j+i}(w)\big)\Big)\right\|_{\H_k} \nonumber\\
&\;\quad+\left\|\sum_{i=0}^{k-j-1}\sqrt\tau L_\tau^{-(k-j)+i}
\Big(B\big(E_{\tau,j,j+i}(v)\big)-B\big(E_{\tau,j,j+i}(w)\big)\Big)\chi_{j+i}\right\|_{\H_k} \nonumber\\
&=: (I) + (I\!\! I) + (I\!\! I\!\! I).
\end{align}
We estimate each of the terms $(I)$, $(I\!\! I)$ and $(I\!\! I\!\! I)$ separately.

By Lemma~\ref{lem:estimatesResolvent} and the trivial fact that $\|v-w\|_{\H_k}=\|v-w\|_{\H_j}$ for all $v,w\in\H_j$, we have
\begin{equation}\label{SPDELipschitzEstimate2}
\begin{aligned}
(I)&\leq \big\|L_\tau^{-1}\big\|_{\mathcal L(D((-A)^\rho))}^{k-j}\|v-w\|_{\mathcal H_k}\\
&\leq (1-\tau\lambda_1)^{-(k-j)}\|v-w\|_{\H_k}\\
&\leq \|v-w\|_{\H_j}.
\end{aligned}
\end{equation}

Concerning the term $(I\!\! I)$ in \eqref{SPDELipschitzEstimate1}, let us first concentrate on the case $\sigma\in (0,1)$. We use the Lipschitz condition on $f$ in Assumption~\ref{AssumptionfBSPDE} and Lemma~\ref{lem:estimatesResolvent}
to obtain
\begin{equation}\label{SPDELipschitzEstimate3}
\begin{aligned}
(I\!\! I)&\leq \sum_{i=0}^{k-j-1}\tau\big\|L_\tau^{-(k-j)+i}(-A)^\sigma\big\|_{\mathcal L(D((-A)^\rho))} \\ %\times \\
&\qquad\qquad\qquad \times C_f^{\text{Lip}}
\big\|E_{\tau,j,j+i}(v)-E_{\tau,j,j+i}(w)\big\|_{\H_{j+i}}\\
&\leq \sum_{i=0}^{k-j-1}\tau\frac{\sigma^\sigma}{(\tau(k-j-i))^\sigma}
C_f^{\text{Lip}}C_{\tau,j,j+i}^{\text{Lip}}\left\|v-w\right\|_{\H_{j}}\\
&\leq C_f^{\text{Lip}}\sum_{i=0}^{k-j-1}\frac{\tau}{(\tau(k-j-i))^\sigma}C_{\tau,j,j+i}^{\text{Lip}}
\left\|v-w\right\|_{\H_{j}}.
\end{aligned}
\end{equation}
For the case that $\sigma\leq 0$ we get with similar arguments 
\begin{equation}\label{SPDELipschitzEstimate3b}
\begin{aligned}
(I\!\! I)&\leq
C_f^{\text{Lip}}\sum_{i=0}^{k-j-1}\frac{\tau (-\lambda_1)^{\sigma}}{(1-\tau\lambda_1)^n}C_{\tau,j,j+i}^{\text{Lip}}
\left\|v-w\right\|_{\H_{j}}\\
&\leq C_f^{\textup{Lip}} (-\lambda_1)^{\sigma} \sum_{i=0}^{k-j-1} \tau C_{\tau,j,j+i}^{\text{Lip}}
\left\|v-w\right\|_{\H_{j}}.
\end{aligned}
\end{equation}

Let us now look at the term $(I\!\! I\!\! I)$ in \eqref{SPDELipschitzEstimate1}. Using the independence of the stochastic increments $\chi_{j+i}$ and the equality in~\eqref{eq:isometry-wn}, we get
\begin{equation*}
\begin{aligned}
(I\!\! I\!\! I)^2\! &= \sum_{i=0}^{k-j-1}\! \tau \mathbb{E}\left\|L_\tau^{-(k-j)+i}
\Big(B\big(E_{\tau,j,j+i}(v)\big)\!-\!B\big(E_{\tau,j,j+i}(w)\big)\Big)\chi_{j+i}\right\|_{D((-A)^{\rho})}^2\\
&\leq \sum_{i=0}^{k-j-1} \tau \big\| L_\tau^{-(k-j)+i} \big\|_{\mathcal{L}(D((-A)^{\rho-\beta- \alpha /2}), D((-A)^{\rho}))}^2 \\
&\qquad\quad\; \times
\mathbb{E}\left\|\Big(B\big(E_{\tau,j,j+i}(v)\big)-B\big(E_{\tau,j,j+i}(w)\big)\Big)\chi_{j+i}\right\|_{D((-A)^{\rho-\beta-\alpha/2})}^2\\
&\leq \sum_{i=0}^{k-j-1} \tau \big\| L_\tau^{-(k-j)+i} \big\|_{\mathcal{L}(D((-A)^{\rho-\beta- \alpha /2}), D((-A)^{\rho}))}^2 \\
&\qquad\quad\; \times
\mathbb{E}\left\| B\big(E_{\tau,j,j+i}(v)\big)-B\big(E_{\tau,j,j+i}(w)\big)\right\|_{\hs(\ell_2;D((-A)^{\rho-\beta-\alpha/2}))}^2.
\end{aligned}
\end{equation*}
Concentrating first on the case $\beta+\alpha/2>0$, we continue by using the Lipschitz condition on $B$ in Assumption~\ref{AssumptionfBSPDE} and Lemma~\ref{lem:estimatesResolvent} to obtain
\begin{equation}\label{SPDELipschitzEstimate4}
\begin{aligned}
(I\!\! I\!\! I)^2&\leq  \sum_{i=0}^{k-j-1}\tau
\frac{(\beta+\alpha/2)^{2\beta+\alpha}}{(\tau(k-j-i))^{2\beta+\alpha}}
\text{Tr}(-A)^{-\alpha}\\
&\qquad\qquad\quad \times (C_{B}^{\text{Lip}})^2\,\mathbb E\left\|E_{\tau,j,j+i}(v)-E_{\tau,j,j+i}(w)\right\|_{D((-A)^\rho)}^2\\
&\leq (C_{B}^{\text{Lip}})^2\operatorname{Tr}(-A)^{-\alpha}\\
&\qquad\qquad\quad \times \sum_{i=0}^{k-j-1}
\frac{\tau}{(\tau(k-j-i))^{2\beta+\alpha}}(C_{\tau,j,j+i}^{\text{Lip}})^2
\|v-w\|_{\H_j}^2.
\end{aligned}
\end{equation}
In the case $\beta+\alpha/2\leq 0$ the same arguments lead to
\begin{equation}\label{SPDELipschitzEstimate4b}
\begin{aligned}
(I\!\! I\!\! I)^2&\leq  \sum_{i=0}^{k-j-1}\tau
\frac{(-\lambda_1)^{2\beta+\alpha}}{(1-\tau\lambda_1)^{2n}}
\text{Tr}(-A)^{-\alpha}\\
&\qquad\qquad \times (C_{B}^{\text{Lip}})^2\mathbb E\left\|E_{\tau,j,j+i}(v)-E_{\tau,j,j+i}(w)\right\|_{D((-A)^\rho)}^2\\
&\leq (C_{B}^{\text{Lip}})^2\operatorname{Tr}(-A)^{-\alpha} (-\lambda_1)^{2\beta+\alpha}\sum_{i=0}^{k-j-1}
\tau(C_{\tau,j,j+i}^{\text{Lip}})^2
\|v-w\|_{\H_j}^2.
\end{aligned}
\end{equation}

Now we have to consider four different cases.

\textbf{Case 1.} $\sigma\in(0,1)$ and $\beta+\alpha/2\in (0,1/2)$.
The combination of \eqref{SPDELipschitzEstimate1}, \eqref{SPDELipschitzEstimate2}, \eqref{SPDELipschitzEstimate3} and \eqref{SPDELipschitzEstimate4} yields
\begin{equation}\label{SPDELipschitzEstimate5}
\begin{aligned}
C_{\tau,j,k}^{\text{Lip}}\leq 1&+C_f^{\text{Lip}}\sum_{i=0}^{k-j-1}\frac{\tau}{(\tau(k-j-i))^\sigma} \, C_{\tau,j,j+i}^{\textup{Lip}}\\
&+C_{B}^{\text{Lip}}(\text{Tr}(-A)^{-\alpha})^{1/2}\left(\sum_{i=0}^{k-j-1}
\frac{\tau}{(\tau(k-j-i))^{2\beta+\alpha}}(C_{\tau,j,j+i}^{\text{Lip}})^2\right)^{1/2}.
\end{aligned}
\end{equation} 
Next, we
estimate the two sums over $i$ on the right hand side of \eqref{SPDELipschitzEstimate5} via H\"{o}lder's inequality. Set
\begin{equation*}
q:=\frac1{\min(1-\sigma,(1-\alpha)/2-\beta)}+2>2.
\end{equation*}
H\"{o}lder's inequality with exponents $q/(q-1)$ and $q$ yields
\begin{equation}\label{SPDELipschitzEstimate6}
\begin{aligned}
\sum_{i=0}^{k-j-1}&\frac{\tau}{(\tau(k-j-i))^\sigma}\,C_{\tau,j,j+i}^{\textup{Lip}}\\
&\leq\left(\sum_{i=0}^{k-j-1}\frac{\tau}{(\tau(k-j-i))^{\frac{\sigma q}{q-1}}}\right)^{\frac{q-1}q}
\left(\sum_{i=0}^{k-j-1}\tau (C_{\tau,j,j+i}^{\text{Lip}})^q\right)^{\frac1q}\\
&\leq\left(\sum_{i=1}^{K}\frac{\tau}{(\tau i)^{\frac{\sigma q}{q-1}}}\right)^{\frac{q-1}q}
\left(\sum_{i=0}^{k-j-1}\tau (C_{\tau,j,j+i}^{\text{Lip}})^q\right)^{\frac1q}\\
&\leq\left(\int_0^Tt^{-\frac{\sigma q}{q-1}}\md t\right)^{\frac{q-1}q}
\left(\sum_{i=0}^{k-j-1}\tau (C_{\tau,j,j+i}^{\text{Lip}})^q\right)^{\frac1q},
%&= C(\alpha,\beta,\sigma)
%\left(\sum_{i=0}^{k-j-1}\tau (C_{\tau,j,j+i}^{\text{Lip}})^q\right)^{\frac1q},
\end{aligned}
\end{equation}
where the integral in the last line is finite since $\frac{\sigma q}{q-1}=\frac{\sigma }{1-1/q}<\frac{\sigma}{1-(1-\sigma)}=1$. Similarly, applying H\"{o}lder's inequality with exponents $q/(q-2)$ and $q/2$,
\begin{equation}\label{SPDELipschitzEstimate7}
\begin{aligned}
\sum_{i=0}^{k-j-1}&
\frac{\tau}{(\tau(k-j-i))^{2\beta+\alpha}}(C_{\tau,j,j+i}^{\text{Lip}})^2\\
&\leq\left(\sum_{i=0}^{k-j-1}
\frac{\tau}{(\tau(k-j-i))^{\frac{(2\beta+\alpha)q}{q-2}}}\right)^{\frac{q-2}q}
\left(\sum_{i=0}^{k-j-1}\tau (C_{\tau,j,j+i}^{\text{Lip}})^q\right)^{\frac2q}\\
&\leq\left(\int_0^Tt^{-\frac{(2\beta+\alpha)q}{q-2}}\md t\right)^{\frac{q-2}q}\left(\sum_{i=0}^{k-j-1}\tau (C_{\tau,j,j+i}^{\text{Lip}})^q\right)^{\frac2q}.
\end{aligned}
\end{equation}
The integral in the last line is finite since
$\frac{(2\beta+\alpha)q}{q-2}=\frac{(2\beta+\alpha)}{1-2/q}
$$<\frac{(2\beta+\alpha)}{1-(1-\alpha-2\beta)}$$=1$.

Combining \eqref{SPDELipschitzEstimate5}, \eqref{SPDELipschitzEstimate6}, \eqref{SPDELipschitzEstimate7} and using the equivalence of norms in $\mathbb R^3$, we obtain
\begin{equation}\label{SPDELipschitzEstimate8}
(C_{\tau,j,k}^{\text{Lip}})^q\leq C_0\left( 1 + \sum_{i=0}^{k-j-1}\tau (C_{\tau,j,j+i}^{\text{Lip}})^q\right),
\end{equation}
with a constant $C_0$ that depends only on $A$, $f$, $B$, $\alpha$, $\beta$, $\sigma$ and $T$.
Since \eqref{SPDELipschitzEstimate8} holds for arbitrary $K\in\N$ and $1\leq j\leq k\leq K$, we can apply a discrete version of Gronwall's lemma and obtain
\begin{equation*}
(C_{\tau,j,k}^{\text{Lip}})^q\leq e^{(k-j)\tau C_0}C_0\leq e^{TC_0}C_0.%,\qquad 0\leq i\leq k-j-1.
\end{equation*}
for all $1\leq j\leq k\leq K$, $K\in\N$ and $\tau=T/K$.
%Since $1\leq j\leq k\leq K$ as well as  $K\in\mathbb N$ have been chosen arbitrary,
It follows that the assertion of the proposition holds in this first case with
\begin{equation*}
C:=(e^{TC_0}C_0)^{1/q}.
\end{equation*}

\textbf{Case 2.} $\sigma\leq 0$ and $\beta+\alpha/2\leq 0$. A combination of \eqref{SPDELipschitzEstimate1} with \eqref{SPDELipschitzEstimate2}, \eqref{SPDELipschitzEstimate3b}, and \eqref{SPDELipschitzEstimate4b} leads to
\begin{equation*}%\label{SPDELipschitzEstimate9}
\begin{aligned}
C_{\tau,j,k}^{\text{Lip}}\leq 1&+
C_f^{\textup{Lip}} (-\lambda_1)^{\sigma} \sum_{i=0}^{k-j-1} \tau C_{\tau,j,j+i}^{\text{Lip}}\\
&+C_{B}^{\text{Lip}} (\text{Tr}(-A)^{-\alpha})^{1/2} (-\lambda_1)^{\beta+\alpha/2}
\Big(\sum_{i=0}^{k-j-1}\tau(C_{\tau,j,j+i}^{\text{Lip}})^2\Big)^{1/2}.
\end{aligned}
\end{equation*}
Applying H\"{o}lder's inequality with exponent $q_2:=2$ to estimate the first sum over $i$ on the right hand side, we get
\begin{equation*}%\label{SPDELipschitzEstimate9}
\begin{aligned}
C_{\tau,j,k}^{\text{Lip}}\leq 1&+
C_f^{\textup{Lip}} (-\lambda_1)^{\sigma} T^{1/2}\Big(\sum_{i=0}^{k-j-1} \tau (C_{\tau,j,j+i}^{\text{Lip}})^2\Big)^{1/2}\\
&+C_{B}^{\text{Lip}} (\text{Tr}(-A)^{-\alpha})^{1/2} (-\lambda_1)^{\beta+\alpha/2}
\Big(\sum_{i=0}^{k-j-1}\tau(C_{\tau,j,j+i}^{\text{Lip}})^2\Big)^{1/2},
\end{aligned}
\end{equation*}
which leads to
\begin{equation*}%\label{SPDELipschitzEstimate9}
(C_{\tau,j,k}^{\text{Lip}})^2\leq C\Big( 1+ \sum_{i=0}^{k-j-1} \tau (C_{\tau,j,j+i}^{\text{Lip}})^2\Big),
\end{equation*}
where the constant $C\in(0,\infty)$ depends only on $A$, $f$, $B$, $\alpha$, $\beta$, $\sigma$ and $T$. As in Case~1, an application of Gronwall's lemma proves the assertion in this second case.

\textbf{Case 3.} $\sigma\in (0,1)$ and $\beta+\alpha/2\leq 0$. In this situation, we combine \eqref{SPDELipschitzEstimate1} with \eqref{SPDELipschitzEstimate2}, \eqref{SPDELipschitzEstimate3} and \eqref{SPDELipschitzEstimate4b} to get
\begin{equation*}%\label{SPDELipschitzEstimate9}
\begin{aligned}
C_{\tau,j,k}^{\text{Lip}}\leq 1&+
C_f^{\text{Lip}}\sum_{i=0}^{k-j-1}\frac{\tau}{(\tau(k-j-i))^\sigma}C_{\tau,j,j+i}^{\text{Lip}}\\
&+C_{B}^{\text{Lip}} (\text{Tr}(-A)^{-\alpha})^{1/2} (-\lambda_1)^{\beta+\alpha/2}
\Big(\sum_{i=0}^{k-j-1}\tau(C_{\tau,j,j+i}^{\text{Lip}})^2\Big)^{1/2}.
\end{aligned}
\end{equation*}
Setting
\begin{equation*}
q_3:=\frac{1}{1-\sigma}+2
\end{equation*}
and following the line of argumentation from the first case with $q_3$ instead of $q$ we reach our goal also in this situation.

\textbf{Case 4.} $\sigma\leq 0$ and $\beta+\alpha/2\in (0,1/2)$. Combine \eqref{SPDELipschitzEstimate1}, \eqref{SPDELipschitzEstimate2}, \eqref{SPDELipschitzEstimate3b} and \eqref{SPDELipschitzEstimate4} to get
\begin{equation*}
\begin{aligned}
C_{\tau,j,k}^{\text{Lip}}\leq 1&
+C_f^{\textup{Lip}} (-\lambda_1)^{\sigma} \sum_{i=0}^{k-j-1} \tau C_{\tau,j,j+i}^{\text{Lip}}\\
&+C_{B}^{\text{Lip}}(\text{Tr}(-A)^{-\alpha})^{1/2}\left(\sum_{i=0}^{k-j-1}
\frac{\tau}{(\tau(k-j-i))^{2\beta+\alpha}}(C_{\tau,j,j+i}^{\text{Lip}})^2\right)^{1/2}.
\end{aligned}
\end{equation*}
Arguing as in the third case with
$$
q_4:=\frac{1}{1/2-(\beta+\alpha/2)}+2
$$
instead of $q_3$, we get the estimate we need to finish the proof.
\end{proof}

We conclude with the proof of our main result.

\begin{proof}[Proof of Theorem~\ref{thm:central-theorem-in-stochastic-section}]
The assertion follows from Proposition~\ref{prop:OverallErrorEstimateSPDE} combined with Proposition~\ref{prop:LipschitzBoundSPDE} by using the elementary estimates
\[\frac 1K\leq\frac 1{K^\delta}=T^{-\delta}\tau^\delta\quad\text{ and }\quad\sum_{j=0}^{K-1}\epsilon_k\leq T\tau^\delta.\qedhere\]
\end{proof}

%% Non-BibTeX users please use

% Addresses
\section*{}
%\ifcase\customizing
%\else
%Corresponding author: Stefan Kinzel\\
%E-Mail: kinzel@mathematik.uni-marburg.de\\
%Phone: +49 6421-28-25483, Fax: +49 6421-28-26945.  \\[2ex]
%\fi
\newpage
\noindent Petru~A.~Cioica, Stephan~Dahlke, Ulrich~Friedrich and Stefan~Kinzel \\
Philipps-Universit{\"a}t Marburg \\
FB Mathematik und Informatik, AG Numerik/Optimierung \\
Hans-Meerwein-Strasse \\
35032 Marburg, Germany \\
\{cioica, dahlke, friedrich, kinzel\}@mathematik.uni-marburg.de \\[2ex]
\noindent Nicolas~D{\"o}hring, Felix~Lindner and Klaus~Ritter \\
TU Kaiserslautern \\
Department of Mathematics, Computational Stochastics Group \\
Erwin-Schr{\"o}dinger-Strasse \\
67663 Kaiserslautern, Germany \\
\{doehring, lindner, ritter\}@mathematik.uni-kl.de \\[2ex]
\noindent Ren{\'e}~L.~Schilling \\
TU Dresden \\
FR Mathematik, Institut f{\"u}r Mathematische Stochastik \\
01062 Dresden, Germany \\
rene.schilling@tu-dresden.de \\[2ex]
\noindent Thorsten~Raasch \\
Johannes Gutenberg-Universit{\"a}t Mainz \\
Institut f{\"u}r Mathematik, AG Numerische Mathematik \\
Staudingerweg 9 \\
55099 Mainz, Germany \\
raasch@uni-mainz.de

\end{document}